\documentclass[twoside,11pt,reqno]{amsart}
\usepackage{natbib}
\usepackage{pstricks}
\usepackage{a4wide}
\usepackage[utf8]{inputenc}
\usepackage{hyperref}
\usepackage{graphicx}
\usepackage{amsmath}
\usepackage{color}
\usepackage{amsmath,amssymb,bbm,enumitem}
\usepackage{epic,eepic}

\usepackage[draft]{fixme}
\FXRegisterAuthor{fr}{afr}{FR}

\newtheorem{theorem}{Theorem}[section]

\newtheorem{definition}{Definition}[section]
\newtheorem{lemma}{Lemma}[section]
\newtheorem{proposition}[theorem]{Proposition}

\theoremstyle{definition}
\newtheorem{remark}{Remark}[section]

\date{\today}
\title[Reduction principle]{Large scale reduction principle
and application to hypothesis testing}

\newcommand{\rmd}{\mathrm{d}}
\newcommand{\rme}{\mathrm{e}}
\newcommand{\rmi}{\mathrm{i}}
\newcommand{\1}{\mathbbm{1}}

\newcommand{\ellset}{\mathcal{L}}
\newcommand{\Irset}{\mathcal{R}}
\newcommand{\K}{\mathbf{L}}

\newcommand{\bh}{\mathbf{h}}
\newcommand{\bW}{\mathbf{W}}
\newcommand{\bx}{\mathbf{x}}

\newcommand{\bS}{\mathbf{S}}

\newcommand{\bkappa}{\boldsymbol{\kappa}}

\author{M. Clausel}
\address{Laboratoire Jean Kuntzmann\\
Université de Grenoble, CNRS\\
F38041 Grenoble Cedex 9}
\email{marianne.clausel@imag.fr}

\author{F. Roueff}
\address{Institut Mines--Telecom, Telecom ParisTech, CNRS LTCI, 46 rue Barrault\\
  75634 Paris Cedex 13, France}
\email{roueff@telecom-paristech.fr}

\author{M.~S. Taqqu}
\address{Departement of Mathematics and Statistics, Boston University,
  Boston, MA 02215, USA}
\email{murad@math.bu.edu}
\thanks{MSC 2010 subject classifications~:~42C40, 60G18, 62M15,60G20,60G22}

\thanks{Keywords~:~long-range
  dependence; long memory; self-similarity; wavelet transform; estimation;
  hypothesis testing.}

\begin{document}
\maketitle
\begin{abstract}
  Consider a non--linear function $G(X_t)$ where $X_t$ is a stationary Gaussian
  sequence with long--range dependence. The usual reduction principle states
  that the partial sums of $G(X_t)$ behave asymptotically like the partial sums
  of the first term in the expansion of $G$ in Hermite polynomials. In the
  context of the wavelet estimation of the long--range dependence parameter, one
  replaces the partial sums of $G(X_t)$ by the wavelet scalogram, namely the
  partial sum of squares of the wavelet coefficients. Is there a reduction
  principle in the wavelet setting, namely is the asymptotic behavior of the
  scalogram for $G(X_t)$ the same as that for the first term in the expansion
  of $G$ in Hermite polynomial? The answer is negative in general. This paper
  provides a minimal growth condition on the scales of the wavelet coefficients
  which ensures that the reduction principle also holds for the scalogram. The
  results are applied to testing the hypothesis that the long-range
  dependence parameter takes a specific value.
\end{abstract}
\tableofcontents
\section{Introduction}\label{s:intro}
Let  $X=\{X_{t}\}_{t\in\mathbb{Z}}$ be a centered stationary
Gaussian process with unit variance and spectral density
$f(\lambda), \lambda \in (-\pi , \pi)$. Such a stochastic process
is said to have {\it
  short memory} or {\it short--range dependence} if $f(\lambda)$ is bounded
around $\lambda=0$ and {\it long memory} or {\it long--range
dependence} if $f(\lambda)\to\infty$ as $\lambda\to0$. We
will suppose that $\{X_{t}\}_{t\in\mathbb{Z}}$ has long memory
with memory parameter $0<d<1/2$, that is,
\begin{equation}\label{e:sdf0}
f(\lambda)\sim |\lambda|^{-2d}f^*(\lambda)\mbox{ as }\lambda \to 0
\end{equation}
where the \emph{short range part} $f^*$ of the spectral density is a bounded
spectral density which is continuous and positive at the origin. The parameter
$d$ is also called the long-range dependence parameter.

A standard
assumption in the semi-parametric setup is
\begin{equation}\label{e:beta-def}
|f^*(\lambda)-f^*(0)|\leq C f^*(0)\,|\lambda|^\beta\,\quad \lambda\in(-\pi,\pi)\;,
\end{equation}
where $\beta$ is some smoothness exponent in $(0,2]$.
This hypothesis is
semi--parametric in nature because the function $f^*$ plays the
role of a ``nuisance function''. It is convenient to set
\begin{equation}\label{e:sdf}
f(\lambda)=|1-\rme^{-\rmi\lambda}|^{-2d}f^*(\lambda),\quad\lambda\in
(-\pi,\pi]\;.
\end{equation}

Consider now a process $\{Y_{t}\}_{t\in\mathbb{Z}}$, such that
\begin{equation}\label{e:DefY}
\left(\Delta^K Y\right)_{t}=G(X_{t}),\quad t\in\mathbb{Z}\;,
\end{equation}
for $K\geq 0$, where $(\Delta Y)_{t}=Y_{t}-Y_{t-1}$, $\{X_t\}_{t\in\mathbb{Z}}$
is Gaussian with spectral density $f$ satisfying~(\ref{e:sdf}) and where $G$ is
a function such that $\mathbb{E}[G(X_{t})]=0$ and
$\mathbb{E}[G(X_{t})^2]<\infty$.  While the process $\{Y_t\}_{t\in\mathbb{Z}}$
is not necessarily stationary, its $K$--th difference $\Delta^K Y_t$ is
stationary. Nevertheless, as in \cite{yaglom:1958} one can speak of the ``generalized spectral density'' of $\{Y_t\}_{t\in\mathbb{Z}}$, which we denote $f_{G,K}$. It is defined as
\begin{equation}\label{e:proof:sd1}
f_{G,K}(\lambda)=|1-e^{-i\lambda}|^{-2K}\,f_G(\lambda)\;,
\end{equation}
where $f_G$ is the spectral density of $\{G(X_t)\}_{t\in\mathbb{Z}}$.

Note that $G(X_t)$ is the output of a non--linear filter $G$ with Gaussian
input. According to the Hermite expansion of $G$ and the value $d$, the time
series $Y$ may be long--range dependent (see
\cite{clausel-roueff-taqqu-tudor-2011a} for more details). We aim at developing
efficient estimators of the memory parameter of such non--linear time series.

Since the 80's many methods for the estimation of the memory parameter have
been developed. Let us cite the Fourier methods developed by Fox and
Taqqu~(\cite{fox:taqqu:1986}) and
Robinson~(\cite{robinson:1995:GPH,robinson:1995:GSE}). Since the 90's, wavelet
methods have become very popular. The idea of using wavelets to estimate the
memory parameter of a time series goes back to~\cite{wornell:oppenheim:1992}
and~\cite{flandrin:1989a,flandrin:1989b,flandrin:1991a,flandrin:1999}.  See
also~\cite{abry:veitch:1998,abry:veitch:flandrin:1998}, \cite{bardet:2002},
\cite{bardet:bibi:jouini:2008}, \cite{bardet:lang:moulines:soulier:2000}. As
shown in~\cite{flandrin:1992}, \cite{abry:veitch:1998}, \cite{veitch:abry:1999}
and~\cite{bardet:2000T} in a parametric context, the memory parameter of a time
series can be estimated using the normalized limit of its {\it scalogram}~(\ref{e:defsnj}), that is
the average of squares of its wavelet coefficients computed at a given scale.
It is well--known that, when considering Gaussian or linear time series, the
wavelet--based estimator of the memory parameter is consistent and
asymptotically Gaussian (see~\cite{moulines:roueff:taqqu:2007:jtsa} for a
general framework in the Gaussian case and \cite{roueff-taqqu-2009} for the
linear case). This result is particulary important for statistical purpose
since it provides confidence intervals for the wavelet--based estimator of the
memory parameter.

The application of wavelet--based methods for the estimation of the memory
parameter of non-Gaussian stochastic processes has been much less treated in
the literature. See~\cite{abry-helgason-pipiras-2011} for some empirical
studies. In \cite{bardet:tudor:2010} is considered the case of the Rosenblatt
process which is a non-Gaussian self-similar process with stationary increments
living in the second Wiener chaos, that is, it can be expressed as a double
iterated integral with respect to the Wiener process. In this case, the
wavelet--based estimator of the memory parameter is consistent but satisfies a
non--central limit theorem. More precisely, conveniently renormalized, the
scalogram which is a sum of squares of wavelet coefficients converges to a
Rosenblatt variable and thus admits a non--Gaussian limit. This result,
surprisingly, also holds for a time series of the form $H_{q_0}(X_t)$ where
$X_t$ is Gaussian with unit variance and $H_{q_0}$ denotes the $q_0$--th
Hermite polynomial with $q_0\geq 2$
(see~\cite{clausel-roueff-taqqu-tudor-2011b}).

The general case $G(X_t)$ is expected to derive from the case
$G=H_{q_0}$. Namely, one could expect that some ``reduction theorem'' analog to
the one of~\cite{taqqu:1975} holds. Recall that the classical reduction theorem
of~\cite{taqqu:1975} states that if $G(X)$ is long--range dependent then the
limit in the sense of finite--dimensional distributions of
$\sum_{k=1}^{[nt]}G(X_k)$ adequately normalized, depends only on the first term
$c_{q_0}H_{q_0}/q_0!$ in the Hermite expansion of $G$. The reduction principle
then states that there exist normalization factors $a_n\to\infty$ as
$n\to\infty$ such that
\[
\frac{1}{a_n}\sum_{k=1}^{[nt]}G(X_k)\qquad\mbox{ and
}\qquad\frac{1}{a_n}\sum_{k=1}^{[nt]}\frac{c_{q_0}}{q_0!}H_{q_0}(X_k)\;,
\]
have the same non--degenerate limit as $n\to\infty$. A reduction principle was established
in~\cite{clausel-roueff-taqqu-tudor-2011a}, Theorem~5.1 for the {\it wavelet coefficients}
of a non--linear time series of the form $G(X_t)$. In applications, the wavelet coefficients are not used directly but only through the scalogram. For example, \cite{fay:moulines:roueff:taqqu:2008survey} use the scalogram to compare Fourier and wavelet estimation methods of the memory parameter. The difficulty is that the scalogram is a quadratic function of the wavelet coefficients involving not only the number of observations but also the scale at which the wavelet coefficients are computed. In practice, however, the scalogram is easy to obtain and one can take advantage of the structure of sample moments to investigate statistical properties. Its use is well--illustrated numerically in~\cite{abry-helgason-pipiras-2011} who consider a number of statistical applications.

The following is a natural question~:

\begin{center} Does a reduction principle hold for the scalogram? \end{center}

In~\cite{clausel-roueff-taqqu-tudor-2013} we illustrated through
different large classes of examples, that the reduction principle for the
scalogram does not necessarily hold and that the asymptotic limit of the
scalogram may even be Hermite process of order greater than $2$. 
It is then important to find sufficient conditions for the reduction principle
to hold. In this case, the normalized limit of the scalogram of the time series
$G(X_t)$ would be the same as the time series $c_{q_0}H_{q_0}(X)/q_0!$ studied
in~\cite{clausel-roueff-taqqu-tudor-2011b} and therefore will be asymptotically
Gaussian if $q_0=1$ and a Rosenblatt random variable if $q_0\geq2$. In
Theorem~\ref{th:main:paper4}, we prove that the reduction
principle holds {\it at large scales}, namely if
\begin{equation}
  \label{eq:rp-cond-intro}
  n_j\ll \gamma_j^{\nu_c}\mbox{ as }j\to\infty\;,
\end{equation}
that is, if the number of wavelet coefficients $n_j$ at scale $j$ (typically
$N2^{-j}$, where $N$ is the sample size) does not grow as fast as the scale
factor $\gamma_j$ (typically $2^j$) to the power $\nu_c$,
as the sample size $N$ and the scale index $j$ go to infinity. The critical exponent $\nu_c$
depends on the function $G$ under consideration and may take the value
$\nu_c=\infty$ for some functions, in which case the reduction principle holds
without any particular growth condition on $\gamma_j$ and $n_j$ besides
$n_j\to\infty$ and $\gamma_j\to\infty$ as $j\to\infty$.

The paper is organized as follows. In Section~\ref{s:LRD}, we introduce
long--range dependence and the scalogram. The main
Theorem~\ref{th:main:paper4}, which states that under
Condition~(\ref{eq:rp-cond-intro}) the reduction principle holds is stated in
Section~\ref{s:main} with the critical exponent $\nu_c$ given in
Section~\ref{sec:critical-exponent} and examples provided in
Section~\ref{sec:examples}. Section~\ref{sec:appl-wavel-estim} contains
statistical applications.  The decomposition of the scalogram in Wiener chaos
is described in Section~\ref{s:decomposition}. That section contains
Theorem~\ref{th:main2:paper4} on which Theorem~\ref{th:main:paper4} is based.
Several proofs are in Section~\ref{s:proofs}. Section~\ref{s:lemmas} contains technical
lemmas. The integral representations are described in
Appendix~\ref{s:appendixA} and the wavelet filters are given in
Appendix~\ref{s:appendixB}. Appendix~\ref{s:appendixC} depicts the multiscale wavelet inference setting.

For the convenience of the reader, in addition to providing a formal proof of a given result, we sometimes describe in a few lines the idea behind the proof.

\section{Long--range dependence and the multidimensional wavelet scalogram}\label{s:LRD}
The centered Gaussian sequence $X=\{X_t\}_{t\in\mathbb{Z}}$ with unit variance and spectral
density~(\ref{e:sdf}) is long--range dependent because $d>0$ and
hence its spectrum explodes at $\lambda=0$.

The long--memory behavior of a time series $Y$ of the form~(\ref{e:DefY}) is
well--known to depend on the expansion of $G$ in Hermite series. Recall that if
$\mathbb{E}[G(X_0)]=0$ and $\mathbb{E}[G(X_0)^2]<\infty$ for
$X_0\sim\mathcal{N}(0,1)$, $G(X)$ can be expanded in Hermite polynomials, that
is,
\begin{equation}\label{e:chaos-exp}
G(X)=\sum_{q=1}^{\infty}\frac{c_q}{q!} H_q(X)\;.
\end{equation}
One sometimes refer to~(\ref{e:chaos-exp}) as an expansion in Wiener chaos.
The convergence of the infinite sum~(\ref{e:chaos-exp}) is in
$L^2(\Omega)$,
\begin{equation}\label{e:cq}
c_q=\mathbb{E}[G(X)H_q(X)]\;,\quad q\geq 1\;,
\end{equation}
and
\begin{equation*}
H_{q}(x)= (-1)^{q}e^{\frac{x^{2}}{2}}\frac{d^{q}}{dx^{q}}\left(
e^{-\frac{x^{2}}{2}}\right)\;,
\end{equation*}
are the Hermite polynomials. These  Hermite polynomials satisfy
$H_0(x)=1,H_1(x)=x,H_2(x)=x^2-1$ and one has
\[
\mathbb{E}[H_q(X)H_{q'}(X)]=
\int_\mathbb{R}H_q(x)H_{q'}(x)\frac{1}{\sqrt{2\pi}}\rme^{-x^2/2}\rmd
x=q!\1_{\{q=q'\}}\;.
\]
Observe that the expansion~(\ref{e:chaos-exp}) starts at $q=1$,
since
\begin{equation}\label{e:c0}
c_0=\mathbb{E}[G(X)H_0(X)]=\mathbb{E}[G(X)]=0\;,
\end{equation}
by assumption. Denote by $q_0\geq 1$ the {\it Hermite rank} of
$G$, namely the index of the first non--zero coefficient in the
expansion~(\ref{e:chaos-exp}). Formally,
\begin{equation}
  \label{e:hermiterank}
q_0=\min\{q\geq1,\,c_q\neq 0\}\;.
\end{equation}
One has then
\begin{equation}\label{e:summability}
\sum_{q=q_0}^{+\infty}\frac{c_q^2}{q!}=\mathbb{E}[G(X)^2]<\infty\;.
\end{equation}

In the special case where $G=H_q$, whether $\{H_q(X_t)\}_{t\in\mathbb{Z}}$ is also long--range dependent
depends on the respective values of $q$ and $d$.  We show
in~\cite{clausel-roueff-taqqu-tudor-2011a}, that the spectral density of
$\{H_q(X_t)\}_{t\in\mathbb{Z}}$ behaves like
$|\lambda|^{-2\delta_+(q)}$ as $\lambda\to 0$, where
\begin{equation}\label{e:ldparamq}
\delta_+(q)=\max(\delta(q),0)\quad\text{where}\quad
\delta(q)=qd-(q-1)/2\;.
\end{equation}
We will also let $\delta_+(0)=\delta(0)=1/2$. For $q\geq 1$, $\delta_+(q)$ is the memory parameter of
$\{H_q(X_t)\}_{t\in\mathbb{Z}}$. It is a non--increasing function of $q$. Therefore, since $0<d<1/2$,
$\{H_q(X_t)\}_{t\in\mathbb{Z}}$, $q\geq 1$, is
long--range dependent\footnote{In our context, the values
$d=1/2-1/(2q)$, $q\geq1$, constitute boundary values which introduce
logarithmic terms and will be omitted for simplicity. See Remark~\ref{rem:log}.} if and only if
\begin{equation}\label{e:dq>0}
\delta(q)>0\Longleftrightarrow d>\frac{1}{2}(1-1/q)\;,
\end{equation}
that is, $d$ must be sufficiently close to $1/2$. Specifically,
for long--range dependence,
\begin{equation}
  \label{eq:qd}
q=1\Rightarrow d>0,\quad q=2\Rightarrow d>1/4,\quad q=3\Rightarrow
d>1/3,\quad q=4\Rightarrow d>3/8\;.
\end{equation}
From another perspective,
\begin{equation}\label{e:dqq>0}
\delta(q)>0\Longleftrightarrow 1\leq q<1/(1-2d)\;,
\end{equation}
and thus $\{H_q(X_t)\}_{t\in\mathbb{Z}}$ is short--range dependent
if $q\geq 1/(1-2d)$.

Recall that the Hermite rank of $G$
is $q_0\geq1$, that is the expansion of $G(X_{t})$ starts at $q_0$. We
always assume that $\{H_{q_0}(X_t)\}_{t\in\mathbb{Z}}$ has long
memory, that is,
\begin{equation}\label{e:longmemorycondition}
q_0 < 1/(1-2d) \;.
\end{equation}
The condition~(\ref{e:longmemorycondition}), with $q_0$ defined as the Hermite
rank~(\ref{e:hermiterank}), ensures such that
$\{Y_t\}_{t\in\mathbb{Z}}=\{\Delta^{-K}G(X_t)\}_{t\in\mathbb{Z}}$ is long-range
dependent with long
memory parameter
\begin{equation}
  \label{eq:d0-def}
d_0=K+\delta(q_0)  \in (K,K+1/2) \; .
\end{equation}
More precisely, we have the following result which also determines a Hölder
condition on the short-range part of the spectral density. This condition shall
involve $q_0$ defined in~(\ref{eq:d0-def}), and, if $G$ is not reduced to $c_{q_0}H_{q_0}/(q_0!)$,
it also involves the index of the second non-vanishing Hermite coefficient denoted by
$$
q_1=\inf\{q>q_0~:~c_q\neq0\}\;.
$$
If there is no such $q_1$ we let $\delta_+(q_1)=0$ in~(\ref{eq:zeta-cond}).
\begin{theorem}\label{thm:spec-density}
Let $Y$ be defined as above. Then the generalized spectral density $f_{G,K}$ of $Y$
can be written as
$$
f_{G,K}(\lambda)=|1-\rme^{-\rmi \lambda}|^{-2 d_0}\;f_{G}^*(\lambda)\;,
$$
where $d_0$ is defined by~(\ref{eq:d0-def}) and $f_{G}^*$ is bounded,
continuous and positive at the origin. Moreover,  for any $\zeta>0$ satisfying
\begin{equation}
  \label{eq:zeta-cond}
  \zeta\leq \min(\beta,2(\delta(q_0)-\delta_+(q_1))\quad\text{and, if
    $q_0\geq2$,}\quad \zeta<2\delta(q_0) \;,
\end{equation}
there exists a constant $C>0$ such that
\begin{equation}\label{e:betatilde-def}
|f_{G}^*(\lambda)-f_{G}^*(0)|\leq C f_{G}^*(0)\,|\lambda|^{\zeta},\,\quad \lambda\in(-\pi,\pi)\;.
\end{equation}
\end{theorem}
\begin{proof}
See Section~\ref{s:proof:spec-density}.
\end{proof}
\noindent{\it Idea behind the proof of Theorem~\ref{thm:spec-density}. }Starting with the regularity of the nuisance function $f^*$ in~(\ref{e:sdf0}), one derives that of $f^*_{H_q}$ and, more generally, that of $f^*_G$, taking advantage of the fact that the terms in the expansion of $G(X)$ in Hermite polynomials are uncorrelated.
\begin{remark}
The exponent $\zeta$ in~(\ref{eq:zeta-cond}) will affect the bias of the mean of the scalogram (see~(\ref{eq:scalogram-exp})). The higher $\zeta$, the lower the bias. Since in ~(\ref{eq:zeta-cond}), $\zeta$ is required to satisfy a non--strict and a strict inequality (if $q_0\geq 2$), we cannot
provide an explicit expression for $\zeta$. However, in most cases one has $q_0=1$ or $\delta_+(q_1)>0$ and hence one can set
$\zeta= \min(\beta,2(\delta(q_0)-\delta_+(q_1)))$ which then satisfies both
inequalities in~(\ref{eq:zeta-cond}).
\end{remark}

Our estimator of the long memory parameter of $Y$ is defined from its wavelet coefficients, denoted by $\{W_{j,k},\,j\geq 0,\,k\in\mathbb{Z}\}$, where $j$
indicates the scale index and $k$ the location. These wavelet coefficients are
defined by
\begin{equation}\label{e:WC}
W_{j,k}=\sum_{t\in\mathbb{Z}}h_j(\gamma_j k-t)Y_{t}\;,
\end{equation}
where $\gamma_j\uparrow \infty$ as $j\uparrow \infty$ is a sequence of
non--negative decimation factors applied at scale index $j$.  The properties of
the memory parameter estimator are directly related to the asymptotic behavior
of the scalogram $S_{n_j,j}$, defined by
\begin{equation}\label{e:defsnj}
S_{n_j,j}=\frac{1}{n_j}\sum_{k=0}^{n_j-1}W_{j,k}^2\;,
\end{equation}
as $n_j\to\infty$ (large sample behavior) and
$j\to\infty$ (large scale behavior). More precisely, we will study the
asymptotic behavior of the sequence
\begin{equation}\label{e:snjm}
\overline{S}_{n_{j+u},j+u}=
S_{n_{j+u},j+u}-\mathbb{E}(S_{n_{j+u},j+u})=\frac{1}{n_{j+u}}\sum_{k=0}^{n_{j+u}-1}\left(W_{j+u,
k}^2-\mathbb{E}(W_{j+u ,k}^{2})\right)\;,
\end{equation}
adequately normalized as $j,n_j\to\infty$.

There are two perspectives. One can consider, as
in~\cite{clausel-roueff-taqqu-tudor-2011a}, that the wavelet coefficients
$W_{j+u,k}$ are processes indexed by $u$ taking a finite number of values. A
second perspective consists in replacing the filter $h_{j}$
in~(\ref{e:WC}) by a multidimensional filter $h_{\ell,j}, \ell=1,\cdots,m$ and
thus replacing $W_{j,k}$ in~(\ref{e:WC}) by
$$
W_{\ell,j,k}=\sum_{t\in\mathbb{Z}}h_{\ell,j}(\gamma_j k-t)Y_{t},\,\ell=1,\cdots,m\;,
$$
(see Appendix~\ref{s:appendixC} for more details). We adopted this second perspective in~\cite{clausel-roueff-taqqu-tudor-2011b,clausel-roueff-taqqu-tudor-2013}
and we also adopt it here since it allows us to compare our results to those
obtained in~\cite{roueff-taqqu-2009} in the Gaussian case.

We use bold faced symbols $\bW_{j,k}$ and $\bh_j$ to
emphasize the multivariate setting and let
\begin{eqnarray*}
\bh_j=\{h_{\ell,j},\,\ell=1,\cdots,m\},\qquad
\bW_{j,k}=\{W_{\ell,j,k},\,\ell=1,\cdots,m\}\;,
\end{eqnarray*}
with
\begin{equation}\label{e:Wjkbold}
\bW_{j,k}=\sum_{t\in\mathbb{Z}}\bh_j(\gamma_j
k-t)Y_t=\sum_{t\in\mathbb{Z}}\bh_j(\gamma_j
k-t)\Delta^{-K}G(X_t),\,j\geq 0,k\in\mathbb{Z}\;.
\end{equation}
We then will study the asymptotic behavior of the sequence
\begin{equation}\label{e:snjbold}
\overline{\mathbf{S}}_{n_j,j}=\frac{1}{n_j}\sum_{k=0}^{n_j-1}\left(\bW_{j,
k}^2-\mathbb{E}[\bW_{j, k}^{2}]\right)\;,
\end{equation}
adequately normalized as $j\to\infty$, where, by convention, in this paper,
\begin{equation}
  \label{eq:conv-square-vector}
\bW_{j, k}^2=
\{W_{\ell,j,k}^2,\,\ell=1,\cdots,m\}\;.
\end{equation}
The squared Euclidean norm of a vector $\bx=[x_1,\dots,x_m]^T$ will be denoted
by $|\bx|^2=x_1^2+\dots+x_m^2$ and the $L^2$ norm of a random vector
$\mathbf{X}$ is denoted by
\begin{equation}
  \label{eq:L2-norm}
\|\mathbf{X}\|_2=\left(\mathbb{E}\left[|\mathbf{X}|^2\right]\right)^{1/2} \;.
\end{equation}

We now summarize the main assumptions of this paper in the following set of conditions.

\bigskip

\noindent{\bf Assumptions A}
$\{\textbf{W}_{j,k},\,j\geq1,k\in\mathbb{Z}\}$ are the
multidimensional wavelet coefficients defined by~(\ref{e:Wjkbold})
, where
\begin{enumerate}[label=(\roman*)]
\item\label{item:X} $\{X_t\}_{t\in\mathbb{Z}}$ is a stationary Gaussian process
with mean $0$, variance $1$ and spectral density $f$
  satisfying~(\ref{e:sdf}).
\item\label{item:LMass} $G$ is a real-valued function whose Hermite
expansion~(\ref{e:chaos-exp})
  satisfies condition~(\ref{e:longmemorycondition}), namely $q_0<1/(1-2d)$, and whose coefficients in the Hermite expansion satisfy the following condition : for any $\lambda>0$
\begin{equation}\label{e:condcv}
c_q=O((q!)^{d} \mathrm{e}^{-\lambda q})\quad\mbox{ as }q\to\infty\;.
\end{equation}
\item\label{item:wave} the wavelet filters $(\bh_j)_{j\geq1}$ and their asymptotic Fourier
  transform $\widehat{\bh}_\infty$ satisfy the standard conditions~\ref{ass:w-ap}--\ref{ass:w-c} with
  $M$ vanishing moments. See details in Appendix~\ref{s:appendixB}.
\end{enumerate}

We shall prove that, provided that the number of vanishing moments of the wavelet is large
enough, these assumptions yield the following general bound for the centered
scalogram.
\begin{theorem}\label{thm:cons}
  Suppose that Assumptions \textbf{A} hold with $M\geq K+\delta(q_0)$. Then for
  any two diverging sequences $(\gamma_j)$ and $(n_j)$, we have, as $j\to\infty$,
  \begin{equation}
    \label{eq:consistency-bound}
\left\|\overline{\mathbf{S}}_{n_j,j}\right\|_2 =
O\left(\gamma_j^{2d_0}n_j^{-(1/2-d)}\right) \;.
  \end{equation}
\end{theorem}
\begin{proof}
Theorem~\ref{thm:cons} is proved in Section~\ref{s:proof:thm:cons}.
\end{proof}
\noindent{\it Idea behind the proof of Theorem~\ref{thm:cons}. }One decomposes $\overline{\mathbf{S}}_{n_j,j}$ further in terms $\mathbf{S}_{n_j,j}^{(q,q',p)}$ as in~(\ref{e:decompSnjLD}) and applies the bounds obtained in part in Proposition~\ref{pro:decompSnjLD:paper3}.

\medskip

It is important to note that Theorem~\ref{thm:cons} holds whatever the relative
growth of $(\gamma_j)$ and $(n_j)$ but it only provides a bound. This bound will
be sufficient to derive a consistent estimator of the long memory parameter
$K+\delta(q_0)$, see Theorem~\ref{thm:cons-wavelet-log-regression} below.

Obtaining a sharp rate of convergence of the centered scalogram and its
asymptotic limit is of primary importance in statistical applications but this
can be quite a complicated task. We exhibit several cases in
\cite{clausel-roueff-taqqu-tudor-2011b,clausel-roueff-taqqu-tudor-2013} that
underline the wild diversity of the asymptotic behavior of the centered
scalogram. In general the nature of the limit depends on the relative growth of
$(\gamma_j)$ and $(n_j)$. We will show, however, that if $n_j\ll\gamma_j^{\nu_c}$,
where $\nu_c$ is a critical exponent, then the reduction principle holds. In
this case, the limit will be either Gaussian or expressed in terms of the
Rosenblatt process which is defined as follows.
\begin{definition}\label{d:HP}
The Rosenblatt process of index $d$ with
\begin{equation}\label{e:qd}
1/4<d<1/2\;,
\end{equation}
is the continuous time process
\begin{equation}\label{e:harmros}
Z_{d}(t)= \int_{\mathbb{R}^{2}}^{\prime\prime} \frac{\rme^{\rmi
(u_1+u_2)\,t}- 1} {\rmi(u_1+u_2)}|u_1|^{-d}|u_2|^{-d}\;
\rmd\widehat{W}(u_1) \rmd\widehat{W}(u_2),\,t\in\mathbb{R}\;.
\end{equation}
\end{definition}
The multiple integral~(\ref{e:harmros}) with respect to the complex-valued Gaussian random measure $\widehat{W}$ is defined in
Appendix~\ref{s:appendixA}. The symbol
$\int_{\mathbb{R}^{2}}^{\prime\prime}$ indicates that one does not
integrate on the diagonal $u_1=u_2$. The integral is well-defined
when~(\ref{e:qd}) holds because then it has finite $L^2$ norm. This
process is self--similar with self-similarity parameter
\[
H=2d \in (1/2,1),
\]
that is for all $a>0$, $\{Z_{d}(at)\}_{t\in\mathbb{R}}$ and $\{a^H
Z_{d}(t)\}_{t\in\mathbb{R}}$ have the same finite--dimensional distributions,
see~\cite{Taq79}. When $t=1$, $Z_d(1)$ is said to have the Rosenblatt
distribution. This distribution is tabulated in \cite{veillette:taqqu:2013}.

\section{Reduction principle at large scales}\label{s:main}
We shall now state the main results and discuss them. They are
proved in the following sections. We use
$\overset{\mathcal{L}}{\longrightarrow}$ to denote convergence in law.

The following result involving the case
\[
G=\frac{c_{q_0}}{q_0!}H_{q_0},\,c_{q_0}\neq 0,\,q_0\geq 1\;,
\]
is proved in Theorem 3.2 of \cite{clausel-roueff-taqqu-tudor-2011b} and will serve as reference~:
\begin{theorem}\label{th:main:paper2}
  Suppose that Assumptions \textbf{A}~\ref{item:X} and
  \textbf{A}~\ref{item:wave} hold with $M\geq K+\delta(q_0)$, where
  $\delta(\cdot)$ is defined in~(\ref{e:ldparamq}). Assume that $Y$ is a
  non--linear time series such that $\Delta^K Y=
  \frac{c_{q_0}}{q_0!}H_{q_0}(X)$, with $q_0\geq 1$ and $q_0<1/(1-2d)$. Define
  the centered multivariate scalogram $\overline{\bS}_{n,j}$ related to $Y$
  by~(\ref{e:snjm}) and let $(n_j)$ and     $(\gamma_j)$ be any two diverging sequences of integers.
  \begin{enumerate}[label=(\alph*)]
  \item\label{item:th-case-gauss-paper2} Suppose $q_0=1$ and that
    $(\gamma_j)$ is a sequence of even integers. Then, as $j\to\infty$,
\begin{equation}
\label{e:q01}
n_j^{1/2}\gamma_j^{-2(d+K)}
\overline{\bS}_{n_j,j}
\overset{\mathcal{L}}{\longrightarrow}
c_1^2\mathcal{N}(0,\Gamma) \; ,
\end{equation}
where $\Gamma$ is the $m\times m$ matrix with entries
\begin{equation}
  \label{eq:GammaDef}
  \Gamma_{\ell,\ell'}= 4\pi (f^*(0))^2 \;  \;
\int_{-\pi}^\pi \left|
\sum_{p\in\mathbb{Z}}
|\lambda+2p\pi|^{-2(K+d)}
[\widehat{h}_{\ell,\infty}\overline{\widehat{h}_{\ell',\infty}}](\lambda+2p\pi)
\right|^2 \, \rmd\lambda\;,\quad 1\leq \ell, \ell'\leq m\; .
\end{equation}
\item\label{item:th-case-rozen-paper2} Suppose $q_0\geq 2$.
Then as $j\to\infty$,
\begin{equation}
\label{e:q02}
n_j^{1-2d} \gamma_{j}^{-2(\delta(q_{0})+K)} \overline{\bS}_{n_j,j}
\overset{\mathcal{L}}{\longrightarrow} \frac{c_{q_0}^2}{(q_0-1)!}\,f^*(0)^{q_0}\,\K_{q_0-1} \,
  Z_d(1) \;,
\end{equation}
where $Z_d(1)$ is the Rosenblatt process in~(\ref{e:harmros}) evaluated at time
$t=1$, $f^*(0)$ is the \emph{short-range spectral density} at zero frequency
in~(\ref{e:sdf0}) and where for any $p\geq 1$, $\K_{p}$ is the deterministic $m$-dimensional
vector $[L_{p}(\widehat{h}_{\ell,\infty})]_{\ell=1,\dots,m}$ with finite
entries defined by
\begin{equation}\label{e:defKp}
L_{p}(g)=\int_{\mathbb{R}^{p}}\frac{|g(u_1+\cdots+u_{p})|^{2}}
{|u_1+\cdots+u_{p}|^{2K}}\; \prod_{i=1}^{p} |u_i|^{-2d}\; \rmd
u_1\cdots\rmd u_{p} \;,
\end{equation}
for any $g:\mathbb{R}\to\mathbb{C}$.
  \end{enumerate}
\end{theorem}
Thus Theorem~\ref{th:main:paper2} states that in the case $G=H_{q_0}$, $q_0\geq
1$ the limit of the scalogram is either Gaussian or has a Rosenblatt
distribution\footnote{This case corresponds to $\ellset=\{0\}$ using the notation introduce in~(\ref{eq:qell-def}) below.}. Our main result
Theorem~\ref{th:main:paper4} states that beyond this simple case, the limits
continue to be either Gaussian or Rosenblatt under fairly general conditions,
involving $n_j$ and $\gamma_j$, namely that $n_j\ll \gamma_j^{\nu_c}$ as
$j\to\infty$ where $\nu_c$ is a positive (possibly infinite) {\it critical
  exponent} given in Definition~\ref{def:nuc}, see Section~\ref{sec:critical-exponent} for details.

\begin{theorem}\label{th:main:paper4}
  Suppose that Assumptions \textbf{A} hold with $M\geq K+\delta(q_0)$, where $\delta(\cdot)$
  is defined in~(\ref{e:ldparamq}) and that
  \begin{equation}
    \label{eq:d-cond}
    d\notin \{1/2-1/(2q)~:~q=1,2,3,\dots\}\;.
  \end{equation}
Define the centered
  multivariate scalogram $\overline{\bS}_{n,j}$ related to $Y$ by~(\ref{e:snjm}). Let $(n_j)$ be any
  diverging sequence of integers such that, as $j\to\infty$,
  \begin{equation}
    \label{eq:reduction-condition}
    n_j\ll\gamma_j^{\nu_c}\;,
  \end{equation}
where $\nu_c$ is given in Definition~\ref{def:nuc} below. Then, the following limits hold depending on the value of $q_0$.
  \begin{enumerate}[label=(\alph*)]
  \item\label{item:th-case-gauss-paper4}
If $q_0=1$ and $\gamma_j$ even, then,  the convergence~(\ref{e:q01}) holds.
\item\label{item:thm-case-rozen} If $q_0\geq 2$, then, the convergence~(\ref{e:q02}) holds.
  \end{enumerate}
\end{theorem}
\begin{proof}
  We shall prove in Theorem~\ref{th:main2:paper4},
  see~(\ref{eq:dominating-reduction-principle}), that, under
  Conditions~(\ref{eq:d-cond}) and~(\ref{eq:reduction-condition}),
  $\overline{\bS}_{n_j,j}$ can be \emph{reduced} to a dominating term
  $\bS_{n_j,j}^{(q_0,q_0,q_0-1)}$ in the sense of the $L^2$
  norm~(\ref{eq:L2-norm}). This dominating term depends only on the term $c_{q_0}H_{q_0}(X)/(q_0!)$ of the expansion of $G(X)$. We can then apply Theorem~\ref{th:main:paper2} to conclude.
\end{proof}
This result extends Theorem~3.1 stated above, where $G$ was restricted to
$G=\frac{c_{q_0}}{q_0!}H_{q_0}$. While extending the result to a much more
general function $G$, Theorem~\ref{th:main:paper4} involves two additional
conditions. Condition~(\ref{eq:d-cond}) is merely here to avoid logarithmic
corrections, see Remark~\ref{rem:log}
below. Condition~(\ref{eq:reduction-condition}) is restrictive only when
$\nu_c$ is finite, in which case it imposes a minimal growth of the analyzing
scale $\gamma_j$ with respect to that of $n_j$.  We say that the {\it reduction
  principle holds at large scales}.  The main interest of having a reduction
principle is to conclude that the same asymptotic analysis is valid as in the
case $G=\frac{c_{q_0}}{q_0!}H_{q_0}$.


\begin{remark}
In practice such a result can be used as follows~:  If $d$, $G$ are both known, $\nu_c$ can be evaluated numerically. We then
  get a practical condition, albeit asymptotic, for the reduction
  principle. See Section~\ref{sec:hypothesis-testing} for an application.
\end{remark}
\begin{remark}\label{rem:log}
  The values $d=1/2-1/(2q)$, $q\geq1$, constitute boundary values which already
  appear in the classical reduction theorem, see~\cite{taqqu:1975}. These
  boundary values also exist in our context. If $d=1/2-1/(2q)$, $q\geq1$, one
  gets similar results but with logarithm terms. In fact, one can show that if
  one drops the restriction~(\ref{eq:d-cond}), then the conclusion of
  Theorem~\ref{th:main:paper4} holds if
  \begin{enumerate}[label=\arabic*)]
  \item\label{item:log-cond1} $n_j\ll\gamma_j^{\nu_c}(\log \gamma_j)^{-4}$.
  \item\label{item:log-cond2} For any $\varepsilon>0$, $\log n_j=o(\gamma_j^\varepsilon)$ and
    $\log \gamma_j=o(n_j^\varepsilon)$ as $j\to\infty$.
  \end{enumerate}
The technical condition~\ref{item:log-cond2} is very weak and
condition~\ref{item:log-cond1} is the same as~(\ref{eq:reduction-condition})
up to a logarithmic correction. We assume~(\ref{eq:d-cond}) for simplicity of
the exposition.
\end{remark}
\begin{remark}
We provided in~\cite{clausel-roueff-taqqu-tudor-2013} several examples for
which different limits are obtained. In these examples one does not
have~(\ref{eq:reduction-condition}) and consequently different terms in the
decomposition in Wiener chaos of the scalogram dominate and provide different
limits. Since the limits are not the same as when $G=H_{q_0}$, the reduction
principle does not hold in these cases.
\end{remark}

\section{Critical exponent}
\label{sec:critical-exponent}

The precise description of the critical exponent given below involves a number of sequences, in particular, the
subsequence of Hermite coefficients $c_q,\,q\geq 1$ that are
\emph{non-vanishing}. We denote this subsequence by
$\{c_{q_\ell}\}_{\ell\in\ellset}$ where $(q_\ell)_{\ell\in\ellset}$ is a
(finite of infinite) increasing sequence of integers such that
\begin{equation}\label{eq:qell-def}
q_{\ell}=\mbox{ index of the }(\ell+1)\mbox{th non--zero
coefficient}\,,\quad \ell\in\ellset\;.
\end{equation}
Thus the indexing set $\ellset$ is a set of consecutive integers starting at 0
with same cardinality as the set of non-vanishing coefficients.
We set
\begin{equation}\label{eq:defJ}
  I_0=\{\ell\in\ellset~:\ell+1\in\ellset, q_{\ell+1}-q_{\ell}=1\} \;,
\end{equation}
that is, $q_{\ell}$ and $q_{\ell+1}$ take consecutive values when
$\ell\in I_0$. The set $I_0$ could be either empty (there are no
consecutive values of $q_\ell$) or not empty. Then we set
\begin{equation}\label{e:l0}
\ell_0=  \begin{cases}
\min(I_0)\geq 0\;,&\text{when $I_0$ is not empty}\;,    \\
\infty\;,&\text{when $I_0$ is  empty}\;.
  \end{cases}
\end{equation}
When $\ell_0$ is finite (that is, $I_0$ is not empty), $q_{\ell_0}$ is the
smallest index $q$ such that two Hermite coefficients $c_q$, $c_{q+1}$ are
non--zero.

We define similarly for any $r\geq 0$
\begin{equation}
  \label{eq:defIrset}
 I_{r} = \{\ell\in\ellset~:~q_{\ell+1}=q_{\ell}+r+1\}\;.
\end{equation}
which involves the terms distant by $r+1$. Finally, we extend the definition of $\ell_0$ in~(\ref{e:l0}) to any $r\geq0$ by
\begin{equation}
  \label{eq:def-ellr}
  \ell_r=\min(I_r)\;.
\end{equation}
 We also define
\begin{equation}
  \label{eq:defIrsetneqemptyset}
\Irset= \{r\geq0~:~I_r\neq\emptyset\text{ and }\delta(r+1)>0\}\;.
\end{equation}
Thus $r\in\Irset$ describes the gaps $r+1$ where $H_{r+1}(X_t)$ is long-range dependent.
Since by~(\ref{e:ldparamq}), $\delta(r+1)>0$ is equivalent to
$r+1<1/(1-2d)$, we have
\begin{equation}
  \label{eq:Irsetisfinite}
\Irset\subset\{0,1,\dots,[1/(1-2d)]-1\}\;.
\end{equation}
Finally, let
\begin{equation}\label{e:Jdset}
J_d=\left\{\ell\in\ellset~:~\delta(q_{\ell+1}-q_{\ell})>0\right\}=\left\{\ell\in\ellset~:~q_{\ell+1}<q_{\ell}+(1-2d)^{-1}\right\}\;,
\end{equation}
where we used the expression for $\delta(q)$ in~(\ref{e:ldparamq}). Note that
\begin{equation}
  \label{eq:Jd-decomp}
J_d = \bigcup_{r\in \Irset} I_{r}\;,
\end{equation}
and thus
\begin{equation}
  \label{eq:J-Irset}
  J_d\neq\emptyset\Longleftrightarrow \Irset\neq\emptyset\;.
\end{equation}
We illustrate these quantities in the following example.

\medskip

\noindent\textbf{Illustration}.
Suppose
\[
G(x)=c_1\,H_1(x)+\frac{c_3}{3!}\,H_3(x)+\frac{c_4}{4!}\,H_4(x)+\frac{c_5}{5!}\,H_5(x)+\frac{c_{24}}{24!}\,H_{24}(x)\;,
\]
where $c_1$, $c_3$, $c_4$, $c_5$ and $c_{24}$ are non-zero constants.
Then
\begin{align*}
&q_0=1,\,q_1=3,\,q_2=4,\,q_3=5,\,q_4=24\mbox{ and }\ellset=\{0,1,2,3,4\}\;,\\
& I_0=\{1,2\},\,I_1=\{0\},\,I_2=\cdots=I_{17}=\emptyset,\,I_{18}=\{3\}\;,\\
&\ell_0=1,\ell_1=0,\ell_{18}=3\;.
\end{align*}
To determine $\Irset$ we need to involve $d$. Here $q_0=1$ so $d$ can take any
value in $(0,1/2)$ to satisfy Condition~(\ref{e:longmemorycondition}) which
guarantees that $G(X)$ is long-range dependent. We need to consider the gaps of
size 1,2 and  19, namely, $r=0,1$ and $18$. Consequently, by~(\ref{e:ldparamq})
and using the fact that $\delta(q)$ is decreasing,
\begin{enumerate}[label=\alph*)]
\item\label{item:6} If $d\in(0,1/4]$, or equivalently $\delta(2)\leq0$, then $\Irset=\{0\}$.
\item\label{item:7} If $d\in(1/4,9/19]$, or equivalently $\delta(2)>0$ and $\delta(19)\leq0$, then $\Irset=\{0,1\}$.
\item\label{item:8} If $d\in(9/19,1/2)$, or equivalently $\delta(19)>0$, then $\Irset=\{0,1,18\}$.
\end{enumerate}
Finally, by~(\ref{eq:Jd-decomp}), we get for $J_d$ the following subsets of
$\ellset$. In Case~\ref{item:6}~: $J_d=I_0=\{1,2\}$, Case~\ref{item:7}~:
$J_d=I_0\cup I_1=\{0,1,2\}$ and Case~\ref{item:8}~: $J_d=I_0\cup I_1\cup I_{18}=\{0,1,2,3\}$.

\medskip

These sets and indices enter in the following definition.
\begin{definition} \label{def:nuc}
 The critical exponent is
\begin{equation*}
\nu_c=\left\{
\begin{array}{l}
  \infty,\mbox{ if }\ellset=\{0\}\;,\\
\\
\infty,\mbox{ if }q_0=1,\,d\leq 1/4\mbox{ and }I_0=\emptyset\;,\\
\\
\frac{d+1/2-2\delta_+(q_{\ell_0})}{d},\mbox{ if }q_0=1,\,d\leq 1/4\mbox{ and }I_0\neq\emptyset\;,\\
\\
\frac{1-2\delta_+(q_1-1)}{2d-1/2},\mbox{ if }q_0=1,\,d> 1/4,\,1\in\ellset\mbox{ and }J_d=\emptyset\;,\\
\\
\min\left(\frac{1-2\delta_+(q_1-1)}{2d-1/2},
\frac{2d+1/2-2\delta_+(q_{\ell_r})-\delta(r+1)}{\delta(r+1)}~:~r\in\Irset\right)
\mbox{ if }q_0=1,\,d> 1/4\mbox{ and }J_d\neq\emptyset\;, \\
\\
\infty,\mbox{ if }q_0\geq 2\mbox{ and }I_0=\emptyset\;,\\
\\
1+\frac{4(\delta(q_0)-\delta_+(q_{\ell_0}))}{1-2d},\mbox{ if }q_0\geq 2\mbox{
  and }I_0\neq\emptyset\;.
\end{array}\right.
\end{equation*}
\end{definition}
The exponent $\nu_c$ depends on $d$ and on the function $G$ through the
expansion coefficient indices $(q_{\ell})_{\ell\in\ellset}$ defined
in~(\ref{eq:qell-def}). In fact one has
\begin{proposition}\label{pro:crit:expo}
  Every possible sequence $(q_{\ell})_{\ell\in\ellset}$ and every value of $d$
  satisfying~(\ref{e:longmemorycondition}) give rise to a $\nu_c\in(0,\infty]$.
\end{proposition}
\begin{proof}
See Section~\ref{s:proof:pro:crit:expo}.
\end{proof}
The value
$\nu_c=\infty$ is the simplest case since then the reduction principle holds
whatever the respective growth rates of the diverging sequences $(n_j)$ and
$(\gamma_j)$ are. This happens for instance when there are no consecutive
non--zeros coefficients ($I_0=\emptyset$) and either $q_0=1$ and $d\leq 1/4$ or
$q_0\geq 2$ (which implies $d>1/4$).

\section{Examples}\label{sec:examples}
In this section, we examine some specific cases of functions $G$. We always
assume that $G$ satisfies Assumption \textbf{A}~\ref{item:LMass}.
\subsection{$G$ is even}
If $G$ is an even function then $q_0\geq2$ and $I_0=\emptyset$ because the
Hermite expansion has only even terms. Hence $\nu_c=\infty$
and the reduction principle applies for any diverging sequences $(n_j)$ and
$(\gamma_j)$.
\subsection{$G$ is odd}
If $G$ is an odd function then we have again $I_0=\emptyset$ since the
Hermite expansion has no even terms. But unlike the even
case, we may have $q_0=1$. If it is not the case, then  $q_0\geq3$ so that $\nu_c=\infty$
and the reduction principle applies for any diverging sequences $(n_j)$ and
$(\gamma_j)$. If $q_0=1$ and $d\leq1/4$, we find again $\nu_c=\infty$. If
$q_0=1$ and $d>1/4$,  the formula of the exponent $\nu_c$ is more involved and
takes various  possible forms, see Section~\ref{sec:324} for one of the
possible cases, namely $I_0=\emptyset$, $q_0=1$ and $\delta(q_1)>0$.

\subsection{$I_0\neq\emptyset$ and  $q_0\geq2$}
This corresponds to the class studied in Section~3.1 of
\cite{clausel-roueff-taqqu-tudor-2013} with the additional condition
$\delta(q_{\ell_0}+1)>0$ (see (3.3) in this reference). Using this additional condition, we have $\delta(q_{\ell_0})>0$ since $\delta(q)$ is decreasing. Hence $\delta_+(q_{\ell_0})=\delta(q_{\ell_0})$ and
$$
\nu_c=1+\frac{4(\delta(q_0)-\delta_+(q_{\ell_0}))}{1-2d}=1+\frac{4(\delta(q_0)-\delta(q_{\ell_0}))}{1-2d}=1+2(q_{\ell_0}-2q_0)\;.
$$
This value of $\nu_c$ corresponds to the
exponent $\nu$ defined in~(3.4) and appearing in Theorem~3.1 of
\cite{clausel-roueff-taqqu-tudor-2013}. This theorem shows that if the opposite
condition to~(\ref{eq:reduction-condition}) holds, namely, $\gamma_j^{\nu_c}\ll
n_j$, then the reduction principle does not apply since the limit is Gaussian
instead of Rosenblatt. We say that the \emph{reduction principle does not apply
  at small scales}. In Theorem~\ref{th:main:paper4}, the reduction principle is
proved even when $\delta_+(q_{\ell_0}+1)=0$, but whether the reduction
principle does apply or not at small scales, namely if $\gamma_j^{\nu_c}\ll
n_j$, remains an open question.

\subsection{$I_0=\emptyset$, $q_0=1$ and $\delta(q_1)>0$}
\label{sec:324}
The expansion of $G$ contains $H_1$ but does not contain any two
consecutive polynomials. This corresponds to the class studied in Section~3.2
of \cite{clausel-roueff-taqqu-tudor-2013} (see (3.8) in this reference).  The
exponent $\nu_c$ simplifies as follows.  First observe that  $\delta(q_1)>0$
implies $\delta_+(q_1-1)>0$, so that $q_1\in\Irset$, and also $\delta(2)>0$
and hence $d>1/4$. We thus need to focus on the term of $\nu_c$  in
Definition~\ref{def:nuc} involving $\min$.  Using~(\ref{e:ldparamq}), for the first
term in the min
\begin{equation}\label{eq:exple-nuc-calc}
  \frac{1-2\delta_+(q_1-1)}{2d-1/2} = \frac{(q_1-1)(1-2d)}{\delta(2)}\;,
\end{equation}
which corresponds to the exponent $\nu_2$ in (3.10) of
\cite{clausel-roueff-taqqu-tudor-2013}.  Now focus on the second term in the
min. Take any $r\in\Irset$ and consider ${\ell_r}$ defined
in~(\ref{eq:def-ellr}). Note that $q_{\ell_r}$ is the smallest Hermite
polynomial index of the expansion of $G$ such that the next one appears after a
gap equal to $r+1$.  There are only two possibilities : (a) either
$q_{\ell_r}=q_0=1$, (b) or $q_{\ell_r}\geq q_1$. In case (a), we have
$r+1=q_{\ell_r+1}-q_{\ell_r}=q_1-1$ and thus
\begin{equation}
  \label{eq:exple-nuc-calc-nu1}
\frac{2d+1/2-2\delta_+(q_{\ell_r})-\delta(r+1)}{\delta(r+1)}=\frac{1/2-\delta(q_1-1)}{\delta(q_1-1)}\;,
\end{equation}
which corresponds to the exponent $\nu_1$ in (3.10) of
\cite{clausel-roueff-taqqu-tudor-2013}. In case (b),
using  $r+1\geq2$ (since $I_0=\emptyset$) and  $q_{\ell_r}\geq q_1$, we get
$$
\frac{2d+1/2-2\delta_+(q_{\ell_r})-\delta(r+1)}{\delta(r+1)} \geq
\frac{2d+1/2-2\delta_+(q_1)-\delta(2)}{\delta(2)} =\frac{q_1(1-2d)}{\delta(2)}>
\frac{(q_1-1)(1-2d)}{\delta(2)}\;,
$$
which already appeared in~(\ref{eq:exple-nuc-calc}). Therefore
with~(\ref{eq:exple-nuc-calc}) and~(\ref{eq:exple-nuc-calc-nu1}) and
Definition~\ref{def:nuc} of $\nu_c$ for $q_0=1$ and $d>1/4$, we get
$$
\nu_c=\min\left(\frac{(q_1-1)(1-2d)}{\delta(2)},\frac{1/2-\delta(q_1-1)}{\delta(q_1-1)}\right)\;,
$$
which corresponds to $\min(\nu_1,\nu_2)$ using the definitions in (3.10) of
\cite{clausel-roueff-taqqu-tudor-2013}. Hence the reduction principle
established in Theorem~\ref{th:main:paper4} under the condition
$n_j\ll\gamma_j^{\nu_c}$ corresponds to the cases $n_j\ll\gamma_j^{\nu_1}$ and
$n_j\ll\gamma_j^{\nu_2}$ of Theorems~3.3 and~3.5 in
\cite{clausel-roueff-taqqu-tudor-2013}, respectively. These two theorems
further show that when the additional condition $\delta(q_{1})>0$ holds the
reduction principle does not hold under the opposite condition
$\gamma_j^{\nu_c}\ll n_j$, illustrating the fact that the reduction principle
may not hold at small scales.

\section{Application to wavelet statistical inference}\label{sec:appl-wavel-estim}
\subsection{Wavelet inference setting}\label{sec:stat-inf-setting}
Suppose that we observe a sample $Y_1,\dots,Y_N$ of $Y$. Recall that $Y$ has long
memory parameter $d_0=K+\delta(q_0)$. In this section, we assume that we are given an unidimensional wavelet filter $g_j$ satisfying Assumptions~\ref{ass:w-ap}--\ref{ass:w-c} in Appendix~\ref{s:appendixB} (see also~(\ref{eq:wav_coeff_def1_dyad}) and~(\ref{eq:multiscale-filters})). Then one can derive the wavelet estimator
\begin{equation}\label{e:hat-d0-def}
\hat d_{0} = \sum_{i=0}^p w_i\log \hat\sigma_{j+i}^2 \;,
\end{equation}
where  $w_0,\dots,w_p$ are well chosen weights and $(\hat\sigma_j^2)_{i\leq j\leq i+p}$ denotes the multiscale scalogram
obtained from $Y_1,\dots,Y_N$,
\begin{equation}\label{e:snjbold-notcentered}
(\hat\sigma_j^2)_{i\leq j\leq i+p}=\mathbf{S}_{n_j,j}=\frac{1}{n_j}\sum_{k=0}^{n_j-1}\mathbf{W}_{j,k}^2\;,
\end{equation}
(see Appendix~\ref{s:appendixC} for more details). In
this setting, $(\gamma_j)$ and $(n_j)$ are specified as follows
\begin{align}\label{eq:wave-set-gammajnj}
\gamma_j=2^j \quad\text{and}\quad n_j=N2^{-j}+O(1)\;.
\end{align}
As usual in this setting the asymptotics are to be understood as $N\to\infty$
with a well chosen diverging sequence $j=j_N$ such that
\begin{align}\label{eq:wave-set-Nj}
\lim_{N\to\infty} N2^{-j} =\infty\;,
\end{align}
and thus $(n_j)$ diverge as $N\to\infty$. We refer to
\cite[Theorem~1]{moulines:roueff:taqqu:2007:jtsa} for the asymptotic behavior of the
mean of the scalogram
\begin{equation}
  \label{eq:scalogram-exp}
  \mathbb{E}\left[\hat\sigma_{j}^2\right] =
C\, 2^{2 d_0 j} \;\left(1+O(2^{-\zeta   j})\right)\;,
\end{equation}
where $C$ is a positive constant and $\zeta$ is an exponent satisfying the
conditions of Theorem~\ref{thm:spec-density}. This relation follows from
Theorem~\ref{thm:spec-density}, provided that $M\geq d_0-1/2$.  Choosing
weights such that $\sum_i w_i=0$ and $\sum_i iw_i=1/(2\log 2)$ then yields
\begin{equation}
  \label{eq:logscalogram-bias}
\sum_{i=0}^p w_i\log \mathbb{E}\left[\hat\sigma_{j+i}^2\right] =
d_0+O(2^{-\zeta j})\;.
\end{equation}

\subsection{Consistency}
We now state a consistency result.
\begin{theorem}\label{thm:cons-wavelet-log-regression}
  Consider the wavelet estimation
  setting~(\ref{e:hat-d0-def})--(\ref{eq:wave-set-Nj}) and suppose that
  Assumptions \textbf{A} hold with $M\geq K+\delta(q_0)$. Then, as $N\to\infty$, $\hat d_0$
  converges to $d_0$ in probability.
\end{theorem}
\begin{proof}
By~(\ref{e:hat-d0-def}), we have
\begin{align}\label{eq:centred-dhat}
\hat d_{0} - \sum_{i=0}^p w_i\log \mathbb{E}\left[\hat\sigma_{j+i}^2\right]
&=\sum_{i=0}^p w_i\, \log\left(1+\frac{\hat\sigma_{j+i}^2
    -\mathbb{E}\left[\hat\sigma_{j+i}^2\right]
  }{\mathbb{E}\left[\hat\sigma_{j+i}^2\right] }\right)  \;.
\end{align}
The numerators in the last ratio are the components of
$\overline{\mathbf{S}}_{n_j,j}$ by~(\ref{e:snjbold-notcentered}).
  By Theorem~\ref{thm:cons} and~(\ref{eq:wave-set-gammajnj}), we have
  $$\overline{\mathbf{S}}_{n_j,j}=O_P\left(\gamma_j^{2d_0}\,n_j^{-(1/2-d)}\right)
=O_P\left(2^{2d_0j}\,(N2^{-j})^{-(1/2-d)}\right)\;.$$
Hence, with~(\ref{eq:centred-dhat}) and~(\ref{eq:scalogram-exp}), we get that
$$
\hat d_{0}-\sum_{i=0}^p w_i\log \mathbb{E}\left[\hat\sigma_{j+i}^2\right] =
O_P((N2^{-j})^{-(1/2-d)})\;.
$$
Applying~(\ref{eq:logscalogram-bias}) then yields
\begin{equation}\label{e:hat-d0-approx}
\hat d_{0} = d_0+O_P\left((N2^{-j})^{-(1/2-d)}\right)+O(2^{-\zeta j})\;.
\end{equation}
The result then follows from~(\ref{eq:wave-set-Nj}).
\end{proof}
\begin{remark}
  We note that this consistency result applies without any knowledge of $G$ or
  $\beta$.
\end{remark}
\subsection{Hypothesis testing}\label{sec:hypothesis-testing}
Consider again a sample $Y_1,\dots,Y_N$ of $Y$ and suppose now that $G$ is
known and has Hermite rank $q_0$.

Denote by $\widetilde{d}_0$ the estimator that would be obtained instead of
$\hat d_0$ if we had $G$ replaced by $c_{q_0} H_{q_0}/(q_0!)$.
We shall apply Theorem~\ref{thm:multiscale:1} and Theorem~\ref{thm:multiscale:2} of Appendix~\ref{s:appendixC}. Theorem~\ref{thm:multiscale:1} (case $q_0=1$) derives from Theorem~2 of~\cite{roueff-taqqu-2009b} and Theorem~\ref{thm:multiscale:2} (case $q_0\geq 2$) derives from Theorem~4.1 of~\cite{clausel-roueff-taqqu-tudor-2011b}. We obtain the following~:~for
conveniently chosen diverging sequences $j=(j_N)$, there
exists some renormalization sequence $(u_N)$ such that as $N\to\infty$,
\begin{equation}\label{eq:dtilde-clt}
u_N(\widetilde{d}_0-d_0)\overset{(\mathcal{L})}{\rightarrow}U(d,K,q_0)\;,
\end{equation}
with
\begin{equation}\label{eq:uNdef}
u_N=\left\{\begin{array}{l}(N2^{-j})^{1/2}\mbox{ if }q_0=1,\\(N2^{-j})^{1-2d}\mbox{ if }q_0\geq 2,\end{array}\right.
\end{equation}
and where $U(d,K,q_0)$ is a centered Gaussian random variable if $q_0=1$ and a
Rosenblatt random variable if $q_0\geq 2$. The precise distribution of
$U(d,K,q_0)$ is given in Theorems~\ref{thm:multiscale:1} and~\ref{thm:multiscale:2}. Beside the chosen wavelet, the distribution of $U$
only depends on $d$, $K$ and $q_0$.

As application of the reduction principle in this setting, we
use~(\ref{eq:dtilde-clt}) to define a statistical test procedure which applies
to a general $G$. Let $d_0^*$ be a given possible value for the true unknown
memory parameter $d_0$ of $Y$ and consider the hypotheses
$$
H_0~:~ d_0=d_0^*\quad\text{against}\quad H_1~:~d_0\in(0, \bar K+1/2)\setminus\{d^*_0\}\;.
$$
Here $\bar K$ denotes a known {\it maximal value} for the true (possibly unknown) integration parameter
$K$. So to insure that the number $M$ of vanishing moments satisfies $M\geq
d_0$, it suffices to impose $M>\bar K$.
Since $G$ is assumed to be known, for the given value $d_0^*$, one can define
the parameters $d^*$, $K^*$ and  $\nu_c^*$ defined as $d$, $K$ and $\nu_c$ by
replacing $d_0$ by  $d_0^*$.

Let $\alpha\in(0,1)$ be a level of confidence. Define the statistical test
\begin{equation}
  \label{eq:test-d}
\delta_s=
\begin{cases}
1 &  \text{ if $|\hat d_0-d^*_0| > s_N(\alpha)$,}\\
0 &  \text{ otherwise.}
\end{cases}
\end{equation}
where $s_N(\alpha)$ is the $(1-\alpha/2)$ quantile of $U(d^*,K^*,q_0)/u_N$.

The following theorem provides conditions for the test $\delta_s$ to be
consistent  with asymptotic level of confidence $\alpha$, namely, that its
power goes to 1 and its first type error goes to $\alpha$  as $N$ goes to $\infty$.
\begin{theorem}\label{thm:test}
  Suppose that Assumptions \textbf{A}(i),(ii) hold with $M>\bar K$ and that the unidimensional wavelet filter $g_j$ satisfies Assumptions~\ref{ass:w-ap}--\ref{ass:w-c}. Assume additionally that~(\ref{eq:d-cond}) holds. Let $j=(j_N)$ be a diverging sequence such
  that~(\ref{eq:wave-set-Nj}) holds. Suppose moreover that, as $N\to\infty$,
  \begin{equation}
    \label{eq:reduction-test}
    N2^{-j}\ll 2^{j\nu_c^*}\;,
  \end{equation}
  and that there exists a positive exponent $\zeta$ satisfying~(\ref{eq:zeta-cond})
  and
  \begin{equation}
    \label{eq:bias-negligibility-condition}
2^{-\zeta j}\ll u_N^{-1}\;,
\end{equation}
with $u_N$ defined as in~(\ref{eq:uNdef}). Then, if~(\ref{eq:reduction-condition}) is
satisfied, $\delta_s$ is a consistent test with asymptotic level of confidence $\alpha$.
\end{theorem}
\begin{remark}
Observe that the different conditions that have to be simultaneously satisfied by $(j_N)$ can be reformulated as follows~:
\begin{itemize}
\item $\lim_{N\to\infty} j_N=\infty$ and $\lim_{N\to\infty} N2^{-j_N}=\infty$.
\item $N2^{-j_N}\ll 2^{j_N\zeta'}$ with
\[
\zeta'=\left\{\begin{array}{l}\min(\nu_c^*,\,2\zeta)\mbox{ if }q_0\geq 1,\\ \min(\nu_c^*,\,\zeta/(1-2d))\mbox{ otherwise}.\end{array}\right.
\]
\end{itemize}
In particular, one can easily check that since $\nu_c^*$ and $\zeta$ are both positive so is $\zeta'$. Hence these conditions are not incompatible.
\end{remark}
\begin{proof}
  See Section~\ref{sec:proof-hyp-test}.
\end{proof}
\noindent\emph{Idea behind the proof of Theorem~\ref{thm:test}.}
Condition~(\ref{eq:reduction-test}) states that $n_j\ll\gamma_j^{\nu_c^*}$ and
will insure that the reduction principle holds under
$H_0$. Condition~(\ref{eq:bias-negligibility-condition}) will ensure that the
bias is negligible under $H_0$. These conditions will allow us through
Relation~(\ref{e:hat-d0-approx-precised}) to transfer the problem to the case
$G(x)=\frac{c_{q_0}}{q_0!}H_{q_0}(x)$ which was treated in
\cite{clausel-roueff-taqqu-tudor-2011b}.
\section{Decomposition in Wiener chaos}\label{s:decomposition}
As in \cite{clausel-roueff-taqqu-tudor-2011a} and \cite{clausel-roueff-taqqu-tudor-2013}, we need the expansion of the scalogram into Wiener chaos. The wavelet coefficients can be expanded in the following way~:
\begin{equation}\label{e:Wjk}
\bW_{j,k}=
\sum_{q=1}^{\infty}\frac{c_q}{q!}\bW_{j,k}^{(q)}\;,
\end{equation}
where $\bW_{j,k}^{(q)}$ is a multiple integral of order $q$. Then, using the
same convention as in~(\ref{eq:conv-square-vector}), we have
\begin{equation}\label{e:sumOFprodwjk}
\bW_{j,k}^2=
\sum_{q=1}^{\infty}\left(\frac{c_q}{q!}\right)^2\;\left(\bW_{j,k}^{(q)}\right)^2
+2 \sum_{q'=2}^\infty\sum_{q=1}^{q'-1} \frac{c_q}{q!}
\frac{c_{q'}}{q'!}\bW_{j,k}^{(q)}\bW_{j,k}^{(q')}\;,
\end{equation}
where the convergence of the infinite sums hold in $L^1(\Omega)$ sense.

\noindent Each $\bW_{j,k}^{(q)}$ is a multiple integral and consequently so is $\overline{\mathbf{S}}_{n_j,j}$ in~(\ref{e:snjbold}). (Basic facts about Multiple integrals and Wiener chaos are recalled in Appendix~\ref{s:appendixA}).

In Proposition~4.2 of~\cite{clausel-roueff-taqqu-tudor-2013}, we gave the following explicit expression of the Wiener chaos expansion of the scalogram.
\begin{proposition}\label{pro:decompSnjLD:paper3}
For all $j$, $\{\bW_{j,k}\}_{k\in\mathbb{Z}}$ is a weakly
stationary sequence. Moreover, for any $j\in\mathbb{N}$,
$\overline{\mathbf{S}}_{n_j,j}$ can be expanded into Wiener chaos as follows
\begin{eqnarray}\label{e:decompSnjLD}
\nonumber\overline{\mathbf{S}}_{n_j,j}& =&
\frac{1}{n_j}\sum_{k=0}^{n_j-1}\bW_{j,k}^2-\mathbb{E}[\bW_{j,0}^2]\\
&=&\sum_{q=1}^{\infty}\left(\frac{c_q}{q!}\right)^2\;\;\sum_{p=0}^{q-1}p!
{{q}\choose{p}}^2\;\mathbf{S}_{n_j,j}^{(q,q,p)}+2\sum_{q'=2}^\infty\sum_{q=1}^{q'-1}\frac{c_q}{q!}
\frac{c_{q'}}{q'!} \sum_{p=0}^{q} \; p!\;
{{q}\choose{p}}{{q'}\choose{p}} \; \mathbf{S}_{n_j,j}^{(q,q',p)}\;,
\end{eqnarray}
where, for all $q,q'\geq 1$ and $0\leq p\leq \min(q,q')$, $\bS_{n_j,j}^{(q,q',p)}$ is of the form
\begin{equation}\label{e:Snjrrp1}
\bS_{n_j,j}^{(q,q',p)}=
\widehat{I}_{q+q'-2p}\left(\mathbf{g}_{n_j,j}^{(q,q',p)}\right)\;,
\end{equation}
and where the infinite sums converge in the $L^1(\Omega)$ sense.
The function $\mathbf{g}_{n_j,j}^{(q,q',p)}(\xi)$,
  $\xi=(\xi_1,\dots,\xi_{q+q'-2p})\in\mathbb{R}^{q+q'-2p}$,
  in~(\ref{e:Snjrrp1}) is defined as follows~:
\begin{equation}\label{e:bg}
\begin{array}{lll}\mathbf{g}_{n_j,j}^{(q,q',p)}(\xi)&=& D_{n_j}(\gamma_j\{\xi_1+\dots+\xi_{q+q'-2p}\})\times \prod_{i=1}^{q+q'-2p}[\sqrt{f(\xi_i)}\1_{(-\pi,\pi)}(\xi_i)]\\
&&\times\;\mathbf{\widehat{\kappa}}_j^{(p)}(\xi_1+\dots+\xi_{q-p},\xi_{q-p+1}+\dots+\xi_{q+q'-2p})
\;,
\end{array}
\end{equation}
where $f$ denotes the spectral density of the underlying Gaussian
process $X$ and for any integer $n$,
\begin{equation}\label{eq:dirichlet}
D_{n}(u)=\frac{1}{n_j}\sum_{k=0}^{n_j-1}\rme^{\rmi k u}=
\frac{1-\rme^{\rmi n_j u}}{n_j(1-\rme^{\rmi  u})}\;,
\end{equation}
denotes the normalized Dirichlet kernel, and for
$\xi_1,\xi_2\in\mathbb{R}$, if $p\neq0$,
\begin{equation}\label{e:intrepKjp1}
{\widehat{\bkappa}}_{j}^{(p)}(\xi_1,\xi_2)=\int_{(-\pi,\pi)^{p}}\left(\prod_{i=1}^{p}f(\lambda_i)\right)\;
\mathbf{\widehat{h}}_j^{(K)}(\lambda_1+\dots+\lambda_p+\xi_1)\overline{\mathbf{\widehat{h}}_{j}^{(K)}
(\lambda_1+\dots+\lambda_p-\xi_2)}\;\rmd^p\lambda\;,
\end{equation}
and, if $p=0$,
\begin{equation}\label{e:intrepKjp1bis}
{\widehat{\bkappa}}_{j}^{(p)}(\xi_1,\xi_2)=
\mathbf{\widehat{h}}_j^{(K)}(\xi_1)\overline{\mathbf{\widehat{h}}_j^{(K)}(\xi_2)}\;.
\end{equation}
\end{proposition}
The random summand $\mathbf{S}_{n_j,j}^{(q,q',p)}$ is expressed in~(\ref{e:Snjrrp1}) as a Wiener--It\^o integral of order $q+q'-2p$ and $q+q'-2p$ will be called the {\it order} of
$\mathbf{S}_{n_j,j}^{(q,q',p)}$.

The limits involved in Theorem~\ref{th:main:paper2} are those given by the term
$\bS_{n_j,j}^{(q_0,q_0,q_0-1)}$ as proved in Propositions~5.3 and 5.4 of~\cite{clausel-roueff-taqqu-tudor-2013}. A sufficient condition to get the
reduction principle is that the other terms are negligible with respect to this
term. Theorem~\ref{th:main:paper4} is then a direct consequence of the
following main result~:
\begin{theorem}\label{th:main2:paper4}
  Suppose that Assumptions \textbf{A} hold with $M\geq K+\delta(q_0)$, where
  $\delta(\cdot)$ is defined in~(\ref{e:ldparamq}) and that~(\ref{eq:d-cond})
  holds.  Define the centered multivariate scalogram $\overline{\bS}_{n,j}$
  related to $Y$ by~(\ref{e:snjm}).  Suppose that $(\gamma_j)$ and $(n_j)$ are
  any diverging sequences of integers. Then
  Condition~(\ref{eq:reduction-condition}) implies, as $j\to\infty$,
  \begin{equation}\label{eq:dominating-reduction-principle}
\left\| \overline{\mathbf{S}}_{n_j,j} - \bS_{n_j,j}^{(q_0,q_0,q_0-1)}\right\|_2
\ll \|\bS_{n_j,j}^{(q_0,q_0,q_0-1)}\|_2\;.
  \end{equation}
\end{theorem}
\begin{proof}
Theorem~\ref{th:main2:paper4} is proved in Section \ref{s:proof:th}.
\end{proof}
\noindent\emph{Idea behind the proof of Theorem~\ref{th:main2:paper4}. }One uses the expansion~(\ref{e:decompSnjLD}). The norms of the relevant terms are bounded in Proposition~\ref{pro:upperB}. We then deduce bounds for the difference $\|\overline{\mathbf{S}}_{n_j,j} -\bS_{n_j,j}^{(q_0,q_0,q_0-1)}\|_2$ in Proposition~\ref{prop:sharp-rates-mainterm}. The main task in the proof of Theorem~\ref{th:main2:paper4} is to show that these bounds are negligible compared to the leading term $\|\bS_{n_j,j}^{(q_0,q_0,q_0-1)}\|_2$ whose asymptotic behavior is also given in Proposition~\ref{prop:sharp-rates-mainterm}.

\medskip
Our results are based on $L^2(\Omega)$ upper bounds of the terms
$\|\bS_{n_j,j}^{(q,q',p)}\|_2$ established in Proposition~5.1
of~\cite{clausel-roueff-taqqu-tudor-2013}. To recall this result, we introduce
some notations.

For any $s\in\mathbb{Z}_+$ and $d\in (0,1/2)$, set
\begin{equation}\label{e:KC}
\Lambda_s(a)= \prod_{i=1}^{s}(a_i!)^{1-2d},\quad\forall
a=(a_{1},\cdots,a_{s})\in\mathbb{N}^s\;.
\end{equation}
For any $q,q',p\geq 0$, define $\alpha,\beta$ and $\beta'$ as follows~:
\begin{align}\label{e:alpha}
\alpha(q,q',p)&=
\begin{cases}
\min\left(1-\delta_+(q-p)-\delta_+(q'-p),1/2\right)
&\text{if }p\neq 0\;,\\
\frac{1}{2}&\text{if }p=0\;,
\end{cases}\\
\label{e:beta}
\beta(q,p)&=\max\left(\delta_+(p)+\delta_+(q-p)-1/2,0\right)\;,\\
\label{e:betap}
\beta'(q,q',p)&=\max\left(2\delta_+(p)+\delta_+(q-p)+\delta_+(q'-p)-1,-1/2\right)\;.
\end{align}
Notice that for any $q\geq 0$, $\beta(q,0)=\delta_+(q)$ and that, by definition
of $\beta,\beta'$, we have, for all $0\leq p\leq q\leq q'$, we have
\begin{equation}  \label{eq:betaP-beta-beta}
\beta'(q,q',p)\leq \beta(q,p)+\beta(q',p)\;.
\end{equation}
 Define
the function $\varepsilon$ on $\mathbb{Z}_+$ as
\begin{equation}
\label{e:varepsilontilde}
\varepsilon(p) =
\begin{cases}
0 &\text{if for any }s\in\{1,\cdots,p\},\,s(1-2d)\neq 1\;,\\
1&\text{if for some
}s\in\{1,\cdots,p\},\,s(1-2d)=1\;.
\end{cases}
\end{equation}
We first recall Proposition~5.1 of \cite{clausel-roueff-taqqu-tudor-2013} where
Part~\ref{item:pro:UBi} corresponds to $p\geq1$ and
Part~\ref{item:pro:UBii} to $p=0$.
\begin{proposition}\label{pro:upperB}
Suppose that Assumptions~A hold.
\begin{enumerate}[label=(\roman*)]
\item\label{item:pro:UBi}
There exists $C>0$ such that for for all $n,\gamma_j\geq2$ and $1\leq q \leq
q'$ and $1\leq p\leq \min(q,q'-1)$,
\begin{equation}\label{e:UB1a}
\begin{array}{lll}
\|\mathbf{S}_{n,j}^{(q,q',p)}\|_{2}&\leq&
C^{\frac{q+q'}{2}}\Lambda_2(q-p,p)^{1/2}\Lambda_2(q'-p,p)^{1/2}\gamma_j^{2K}\\
&&\times [n_j^{-\alpha(q,q',p)}\gamma_j^{\beta'(q,q',p)}+n_j^{-1/2}\gamma_j^{\beta(q,p)+\beta(q',p)}]\\
&&\times\left(\log n_j\right)^{\varepsilon(q+q'-2p)}(\log\gamma_j)^{3\varepsilon(q')}\;.
\end{array}
\end{equation}
\item\label{item:pro:UBii}
Assume that $M\geq K+\max(\delta_+(q),\delta_+(q'))$. Then
there exists some $C>0$ such that
for all $n,\gamma_j\geq2$ and $1\leq q \leq q'$,
\begin{equation}\label{e:UB1b}
\|\mathbf{S}_{n_j,j}^{(q,q',0)}\|_{2}\leq
C^{\frac{q+q'}{2}}\Lambda_1(q)^{1/2}\Lambda_1(q')^{1/2}
n_j^{-1/2}\gamma_j^{2K+\delta_+(q)+\delta_+(q')}
\left(\log\gamma_j\right)^{\varepsilon(q')}\;.
\end{equation}
\end{enumerate}
\end{proposition}
Note that under Condition~(\ref{eq:d-cond}) we have $\varepsilon(p)=0$ for all
$p\geq1$ in~(\ref{e:varepsilontilde}). Thus the logarithmic terms vanish
in~(\ref{e:UB1a}) and~(\ref{e:UB1b}). Moreover, if $p=0$ then $\Lambda_2(q,0)=\Lambda_1(q)$, $\alpha(q,q',0)=1/2$, $\beta(q,0)=\delta_+(q)$ and
$\beta'(q,q',0)=\delta_+(q)+\delta_+(q')$. Therefore, if
Condition~(\ref{eq:d-cond}) holds, the bounds~(\ref{e:UB1a})
and~(\ref{e:UB1b}) imply the following common bound
\begin{multline}\label{e:UB1aANDb}
\|\mathbf{S}_{n,j}^{(q,q',p)}\|_{2}\leq
C^{\frac{q+q'}{2}}\Lambda_2(q-p,p)^{1/2}\Lambda_2(q'-p,p)^{1/2}\gamma_j^{2K}\\
\times [n_j^{-\alpha(q,q',p)}\gamma_j^{\beta'(q,q',p)}+n_j^{-1/2}\gamma_j^{\beta(q,p)+\beta(q',p)}]\;.
\end{multline}
Consider now the decomposition
$$
 \overline{\mathbf{S}}_{n_j,j} =
\bS_{n_j,j}^{(q_0,q_0,q_0-1)} + \left( \overline{\mathbf{S}}_{n_j,j} -
  \bS_{n_j,j}^{(q_0,q_0,q_0-1)}\right) \;.
$$
The following result provides the sharp rate of the first term and a bound on
the second one, relying on Wiener chaos decomposition~(\ref{e:decompSnjLD}).
\begin{proposition}\label{prop:sharp-rates-mainterm}
Assume that Assumptions~\textbf{(A)} hold with  $M\geq K+\delta_+(q_0)$ and
suppose that Condition~(\ref{eq:d-cond}) holds. Let
$(n_j)$ and $(\gamma_j)$ be any diverging sequences.
Then, there exists a positive constant $C$ such that, for all $j\geq1$,
\begin{equation}
  \label{eq:bound-reduction-remainder}
  \left\|\overline{\mathbf{S}}_{n_j,j} -
  \bS_{n_j,j}^{(q_0,q_0,q_0-1)}\right\|_2 \leq
C\; \gamma_j^{2K}\sup_{(q,q',p)\in \mathcal A_0}[n_j^{-\alpha(q,q',p)}\gamma_j^{\beta'(q,q',p)}+n_j^{-1/2}\gamma_j^{\beta(q,p)+\beta(q',p)}]\;,
\end{equation}
where we denote
\begin{equation}\label{e:def:A0}
\mathcal A_0 = \{(q,q',p)~:~1\leq q\leq q',\,0\leq p\leq
\min(q,q'-1),\,c_q\times c_{q'}\neq0\}\setminus\{(q_0,q_0,q_0-1)\}\;.
\end{equation}
Moreover, the two following assertions hold~:
\begin{enumerate}[label=(\roman*)]
\item If  $q_0=1$, as $j\to\infty$,
\begin{equation}
  \label{eq:q0eq1-rate-leading}
  \|\bS_{n_j,j}^{(q_0,q_0,q_0-1)}\|_2=  \|\bS_{n_j,j}^{(1,1,0)}\|_2\sim C\;
n_j^{-1/2}\gamma_j^{2(d+K)}\;,
\end{equation}
where $C$ is a positive constant.
\item If $q_0\geq 2$, as $j\to\infty$,
\begin{equation}
  \label{eq:q0neq1-rate-leading}
  \|\bS_{n_j,j}^{(q_0,q_0,q_0-1)}\|_2 \sim C
n_j^{-1+2d}\gamma_j^{2(\delta(q_0)+K)}\;,
\end{equation}
where $C$ is a positive constant.
\end{enumerate}
\end{proposition}
\begin{proof}
By Proposition~\ref{pro:decompSnjLD:paper3} and~(\ref{e:def:A0}), applying the Minkowski inequality,
we have
$$
\left\| \overline{\mathbf{S}}_{n_j,j} - \bS_{n_j,j}^{(q_0,q_0,q_0-1)}\right\|_2
\leq 2\sum_{(q,q',p)\in \mathcal A_0}\frac{|c_q|}{q!}
\frac{|c_{q'}|}{q'!}\; p!\;
{{q}\choose{p}}{{q'}\choose{p}} \|\bS_{n_j,j}^{(q,q',p)}\|_2\;.
$$
The bound~(\ref{e:UB1aANDb}) implies that
\begin{align*}
&\sum_{(q,q',p)\in \mathcal A_0}\frac{|c_q|}{q!}
\frac{|c_{q'}|}{q'!} \; p!\;
{{q}\choose{p}}{{q'}\choose{p}} \|\bS_{n_j,j}^{(q,q',p)}\|_2\\
&\leq \left(\sum_{(q,q',p)\in \mathcal A_0}\frac{|c_q|}{q!}
\frac{|c_{q'}|}{q'!}  \; p!\;
{{q}\choose{p}}{{q'}\choose{p}}C^{\frac{q+q'}{2}}\Lambda_2(q-p,p)^{1/2}\Lambda_2(q'-p,p)^{1/2}\right)\\
&\times \gamma_j^{2K}\sup_{(q,q',p)\in \mathcal A_0}[n_j^{-\alpha(q,q',p)}\gamma_j^{\beta'(q,q',p)}+n_j^{-1/2}\gamma_j^{\beta(q,p)+\beta(q',p)}]\;.
\end{align*}
By Lemma~8.6 of \cite{clausel-roueff-taqqu-tudor-2013}, the last two displays yield~(\ref{eq:bound-reduction-remainder}).

We now prove~(\ref{eq:q0eq1-rate-leading}) and~(\ref{eq:q0neq1-rate-leading}).
First consider the case where $q_0=1$. This asymptotic
equivalence~(\ref{eq:q0eq1-rate-leading}) is related to the convergence~(\ref{e:q01}) and follows from its proof, see
e.g. \cite{moulines:roueff:taqqu:2007:jtsa}. Since $q_0=1$, we have $c_1\neq
0$. Moreover in Condition~\ref{ass:w-c} on the wavelet filters recalled in
Appendix~\ref{s:appendixB}, $\widehat{h}_{\ell,\infty}$ are functions that are
non-identically zero and which are continuous as locally uniform limits of
continuous functions. Therefore $\sum_\ell\Gamma_{\ell,\ell}^2>0$ and we
get~(\ref{eq:q0eq1-rate-leading}).

Now consider the case where $q_0\geq2$. The
bound~(\ref{eq:q0neq1-rate-leading}) is then related to
Theorem~\ref{th:main:paper4}(\ref{item:thm-case-rozen}) where the weak
convergence is stated and follows from its proof, see~\cite{clausel-roueff-taqqu-tudor-2011b}.
\end{proof}
\section{Proofs}\label{s:proofs}
\subsection{Proof of Theorem~\ref{thm:spec-density}}\label{s:proof:spec-density}
The generalized spectral density $f_{G,K}$ of $Y$ is related to the spectral density $f_G$ of $G(X)$ by~(\ref{e:proof:sd1}). By definition of $d_0$, the result shall then follow if we prove the existence of a bounded function $f_{G}^*$ such that
\begin{equation}\label{e:proof:sd2}
f_{G}(\lambda)=|1-\rme^{-\rmi \lambda}|^{-2 \delta(q_0)}\;f_{G}^*(\lambda)\;,
\end{equation}
and satisfying all the properties stated in Theorem~\ref{thm:spec-density}.

We now prove~(\ref{e:proof:sd2}). To this end, we consider the following
decomposition of $G(X)$ as the sum of two uncorrelated processes,
$$
G(X)=G_1(X)+G_2(X)=:\sum_{1\leq q<1/(1-2d)}\frac{c_q}{q!} H_q(X)+\sum_{q\geq1/(1-2d)}
\frac{c_q}{q!} H_q(X)\;.
$$
The proof of Proposition~6.2 in~\cite{clausel-roueff-taqqu-tudor-2011a}
shows that $G_2(X)$ admits a bounded spectral
density $f_{G_2}$. We first consider the case where $G_1$ reduces to the term $c_{q_0}H_{q_0}/q_0!$. Since the two processes $H_{q_0}(X)$ and $G_2(X)$ are uncorrelated, one has
\[
f_G(\lambda)=\frac{c_{q_0}^2}{q_0!}f_{H_{q_0}}(\lambda)+f_{G_2}(\lambda)\;.
\]
We can then set
\[
f^*_G(\lambda)=\frac{c_{q_0}^2}{q_0!}f_{H_{q_0}}^*(\lambda)+|1-\rme^{-\rmi\lambda}|^{2\delta(q_0)}f_{G_2}(\lambda)\;.
\]
Let us check $f^*_G$ has the properties stated in the theorem. Relation~(\ref{e:proof:sd2}) follows from the definition of $f_{G}^*$ and $f_{H_{q_0}}^*$. To prove the other properties stated in Theorem~\ref{thm:spec-density}, we distinguish the two cases $q_0=1$ and $q_0\geq 2$. If $q_0=1$, $f_{H_{q_0}}=f$, $f^*_{H_{q_0}}=f^*$ and then
$\zeta\leq \beta$, one has
\begin{equation}\label{eq:spec-density0}
|f_{H_{q}}^*(\lambda)-f_{H_{q}}^*(0)|\leq C|\lambda|^\zeta\;,
\end{equation}
for some $C>0$. If $q_0\geq 2$, Lemma~\ref{lem:spec-density} yields that there exists a bounded function $f_{H_{q_0}}^*$ such that
\begin{equation}\label{eq:spec-density1}
f_{H_{q}}(\lambda)=|1-\rme^{-\rmi\lambda}|^{-2\delta(q)}f_{H_{q}}^*(\lambda)
\end{equation}
Moreover for any $\zeta\in (0,2\delta(q))$ such that $\zeta\leq \beta$, one has
\begin{equation}\label{eq:spec-density2}
|f_{H_{q}}^*(\lambda)-f_{H_{q}}^*(0)|\leq C|\lambda|^\zeta\;,
\end{equation}
for some $C>0$. In any case, the boundedness of $f_{G_2}$ and the properties of $f^*_{H_{q_0}}$ (equation~(\ref{eq:spec-density0}) if $q_0=1$ or~(\ref{eq:spec-density1}),~(\ref{eq:spec-density2}) if $q_0\geq 2$) then imply that~(\ref{e:betatilde-def}) holds in the case $G_1=c_{q_0}H_{q_0}/q_0!$, that is if $\delta_+(q_1)=0$.

We now deal with the case where $H_{q_1}$ has also long memory, namely $\delta_+(q_1)>0$. Since the terms $H_q(X)$ for $q<1/(1-2d)$ are all pairwise uncorrelated, the spectral density of long-range dependent part $G_1(X)$ reads as follows
\[
f_{G_1}(\lambda)=\sum_{1\leq q<1/(1-2d)}\frac{c_{q}^2}{q!}f_{H_{q}}(\lambda)\;.
\]
We now apply Equation~(\ref{e:fstar-Hq}) of Lemma~\ref{lem:spec-density} successively to each $q<1/(1-2d)$. Hence
\[
f_{G_1}(\lambda)=|1-\rme^{-\rmi\lambda}|^{-2\delta(q_0)}\left[\sum_{1\leq q<1/(1-2d)}\frac{c_{q}^2}{q!}|1-\rme^{-\rmi\lambda}|^{2\delta(q_0)-2\delta(q)}f_{H_{q}}^*(\lambda)\right]\;.
\]
Since $f_G=f_{G_1}+f_{G_2}$, we then get~(\ref{e:proof:sd2}) with
\[
f^*_G(\lambda)=\left[\sum_{1\leq q<1/(1-2d)}\frac{c_{q}^2}{q!}|1-\rme^{-\rmi\lambda}|^{2\delta(q_0)-2\delta(q)}f_{H_{q}}^*(\lambda)\right]+|1-\rme^{-\rmi\lambda}|^{2\delta(q_0)}f_{G_2}(\lambda)\;.
\]
Since $|1-\rme^{-\rmi\lambda}|=0$ for $\lambda=0$, we have $f^*_G(0)=c_{q_0}^2 f_{H_{q_0}}^*(0)/q_0!$.
We now prove that under Condition~(\ref{eq:zeta-cond}) on $\zeta$, we get~(\ref{e:betatilde-def}). Indeed,
\begin{eqnarray*}
|f_{G}^*(\lambda)-f^*_G(0)|&\leq \frac{c_{q_0}^2}{q_0!}|f_{H_{q_0}}^*(\lambda)-f_{H_{q_0}}^*(0)|+\left[\sum_{q_1\leq q<1/(1-2d)}\frac{c_{q}^2}{q!}|1-\rme^{-\rmi\lambda}|^{2\delta(q_0)-2\delta(q)}f_{H_{q}}^*(\lambda)\right]\\
&+
|1-\rme^{-\rmi\lambda}|^{2\delta(q_0)}f_{G_2}(\lambda)\;.
\end{eqnarray*}
Using the boundedness of $f^*_{H_q}$ for any $q\geq q_1$, we deduce that for some $C>0$ and any $q\geq q_1$,
\begin{equation}\label{e:1}
|1-\rme^{-\rmi\lambda}|^{2\delta(q_0)-2\delta(q)}f_{H_{q}}^*(\lambda)\leq C|\lambda|^{2\delta(q_0)-2\delta(q_1)}\;,
\end{equation}
whereas by Lemma~\ref{lem:spec-density} applied with $q=q_0$, we deduce that for any $\zeta\in (0,2\delta(q_0))$ such that $\zeta\leq \beta$, one has
\begin{equation}\label{e:2}
|f_{H_{q_0}}^*(\lambda)-f_{H_{q_0}}^*(0)|\leq L|\lambda|^\zeta\;.
\end{equation}
We now combine~(\ref{e:1}) and~(\ref{e:2}) and deduce that for any $\zeta\in (0,2\delta(q_0))$ such that $\zeta\leq \min(\beta,2\delta(q_0)-2\delta(q_1))$, one has
\begin{equation}\label{e:3}
|f_{G}^*(\lambda)-f_{G}^*(0)|\leq L'|\lambda|^\zeta\;,
\end{equation}
for some $L'>0$.
\subsection{Proof of Theorem~\ref{thm:cons}}\label{s:proof:thm:cons}
 The bound~(\ref{eq:consistency-bound}) follows the same lines as the proof of
  Proposition~\ref{prop:sharp-rates-mainterm}. It is a consequence of
  Proposition~\ref{pro:decompSnjLD:paper3}, Proposition~\ref{pro:upperB},
  Lemma~8.6 of \cite{clausel-roueff-taqqu-tudor-2013} and of the following bounds~:
\begin{align*}
  \alpha(q,q',p)\geq 1/2 -d &\quad\text{and equality implies $q'=q+1$ and $p=q$} \;,\\
  \beta(q,p)\leq \delta(q_0)  &\quad\text{and equality implies $q=q_0$}  \;,\\
  \beta'(q,q',p)\leq 2\delta(q_0)  &\quad\text{and equality implies $q=q'=q_0$} \;.
\end{align*}
The two first bounds follow from Lemma~8.3 in~\cite{clausel-roueff-taqqu-tudor-2013}, and the last one
from~(\ref{eq:betaP-beta-beta}). The equality cases are used to get rid of the
logarithmic corrections appearing in~(\ref{e:UB1a}) and~(\ref{e:UB1b})
since $q'=q+1$ and $p=q$  imply $\varepsilon(q+q'-2p)=\varepsilon(1)=0$ and $q=q'=q_0$
implies $\varepsilon(q')=\varepsilon(q_0)=0$.
This concludes the proof.
\subsection{Proof of Proposition~\ref{pro:crit:expo}}\label{s:proof:pro:crit:expo}
We want to show we always have $\nu_c>0$.  By definition of $\delta$ and $\delta_+$ in~(\ref{e:ldparamq}), we have
$\delta(q)\leq\delta(1)=d$ for all $q\geq1$, and since $d>0$, $\delta_+(q)\leq
d$. With $d<1/2$, this implies that
$d+1/2-2\delta_+(q_{\ell_0})\geq 1/2-d>0$ and thus the third line of the definition of $\nu_c$
is positive. For the same reason, $1-2\delta_+(q_1-1)\geq 1-2d$ and  the fourth line of the definition of $\nu_c$
is positive. In addition, for any $r\geq0$, $2d+1/2-2\delta_+(q_{\ell_r})-\delta(r+1)\geq 1/2-d$ and
$\delta(r+1) < 1/2$ so that, for any $r\in\Irset$,
$$
\frac{2d+1/2-2\delta_+(q_{\ell_r})-\delta(r+1)}{\delta(r+1)}\geq 1-2d >0 \; ,
$$
which ensures that the quantities inside the min are uniformly lower-bounded by a positive value. Finally, for the last line, we separate the cases
$\delta_+(q_{\ell_0})=0$ and $\delta_+(q_{\ell_0})=\delta(q_{\ell_0})$.
In the first case, we have
$4(\delta(q_0)-\delta_+(q_{\ell_0}))=4\delta(q_0)=2-2q_0(1-2d)$ and so
$$
\frac{4(\delta(q_0)-\delta_+(q_{\ell_0}))}{1-2d}=2(1/(1-2d)-q_0)>0 \;,
$$
as a consequence of~(\ref{e:longmemorycondition}).
In the second case, we have
$4(\delta(q_0)-\delta_+(q_{\ell_0}))\geq4(\delta(q_0)-\delta(q_{\ell_0}))=2(q_{\ell_0}-q_0)(1-2d)$ and so
$$
\frac{4(\delta(q_0)-\delta_+(q_{\ell_0}))}{1-2d}\geq  2(q_{\ell_0}-q_0)\geq0 \;,
$$
which is non-negative by definition of $q_{\ell_0}$. Hence the last line
defining $\nu_c$ is at least one, hence is positive, which concludes the proof.
\subsection{Proof of Theorem~\ref{th:main2:paper4}}\label{s:proof:th}
By Proposition~\ref{prop:sharp-rates-mainterm}, it is sufficient to show that
the right-hand side of~(\ref{eq:bound-reduction-remainder}) is negligible with
respect to the right-hand side of~(\ref{eq:q0eq1-rate-leading}) if  $q_0=1$
or  to the right-hand side of~(\ref{eq:q0neq1-rate-leading}) if $q_0\geq2$, that is, respectively,
\begin{align}
  \label{eq:remains-q0isone}
& \lim_{j\to\infty}  n_j^{1/2}\gamma_j^{-2d}\sup_{(q,q',p)\in \mathcal
  A_0}[n_j^{-\alpha(q,q',p)}\gamma_j^{\beta'(q,q',p)}+n_j^{-1/2}\gamma_j^{\beta(q,p)+\beta(q',p)}] =0\;,\\
  \label{eq:remains-q0geq2}
&\lim_{j\to\infty}  n_j^{1-2d}\gamma_j^{-2\delta(q_0)}\sup_{(q,q',p)\in \mathcal
  A_0}[n_j^{-\alpha(q,q',p)}\gamma_j^{\beta'(q,q',p)}+n_j^{-1/2}\gamma_j^{\beta(q,p)+\beta(q',p)}] =0\;.
\end{align}

We now distinguish the two cases $q_0=1$, $q_0\geq 2$.
\subsubsection{Proof of Theorem~\ref{th:main2:paper4} in the case
  $q_0=1$}\label{s:proof:th:case1}

In this case, we need to show that Condition~(\ref{eq:reduction-condition})
implies~(\ref{eq:remains-q0isone}).

By Lemma~8.3 (4) in \cite{clausel-roueff-taqqu-tudor-2013}, we have, for all
$0\leq p\leq q\leq q'$,  $\beta(q,p)+\beta(q',p)\leq
\delta_+(q)+\delta_+(q')$. We may thus write
$$
\sup_{(q,q',p)\in \mathcal
  A_0}\gamma_j^{\beta(q,p)+\beta(q',p)}\leq\sup_{(q,q',p)\in \mathcal
  A_0}\gamma_j^{\delta_+(q)+\delta_+(q')}\leq \gamma_j^{\sup\left\{\delta^+(q)+\delta^+(q')~:~1\leq q\leq q',\,(q,q')\neq(1,1)\right\}}\;,
$$
since for $q_0=1$, the triplet $(q_0,q_0,q_0-1)=(1,1,0)$ is excluded from
$\mathcal A_0$.  Using Lemma~\ref{lem:peq0}, we obtain, as $j\to\infty$,
\begin{equation}
  \label{eq:term2-q0is1}
\sup_{(q,q',p)\in \mathcal A_0}n_j^{-1/2}\gamma_j^{\beta(q,p)+\beta(q',p)}=o\left(n_j^{-1/2}\gamma_j^{2d}\right)\;.
\end{equation}
Inserting this in~(\ref{eq:remains-q0isone}), we only need to show that~(\ref{eq:reduction-condition}) implies
\begin{equation}  \label{eq:term1-q0is1}
\lim_{j\to\infty}  n_j^{1/2}\gamma_j^{-2d}\sup_{(q,q',p)\in \mathcal
  A_0}n_j^{-\alpha(q,q',p)}\gamma_j^{\beta'(q,q',p)} =0\;.
\end{equation}
Observe that, by definition, $\alpha(q,q',p)\leq1/2$.
We shall therefore partition $\mathcal A_0$ into $\mathcal A_0=\mathcal A_1\cup\mathcal A_2$, where
\begin{align*}
\mathcal A_1&=\{(q,q',p)\in\mathcal A_0~:~\alpha(q,q',p)=1/2\}\\
\mathcal A_2&=\{(q,q',p)\in\mathcal A_0~:~\alpha(q,q',p)<1/2\}\;.
\end{align*}
Since $\alpha(q,q',p)=1/2$ for $(q,q',p)\in\mathcal A_1$, we get
with~(\ref{eq:betaP-beta-beta}) that
$$
\sup_{(q,q',p)\in \mathcal A_1}
  n_j^{-\alpha(q,q',p)}\gamma_j^{\beta'(q,q',p)}
\leq  \sup_{(q,q',p)\in \mathcal A_0}
  n_j^{-1/2}\gamma_j^{\beta(q,p)+\beta(q',p)}=o\left(n_j^{-1/2}\gamma_j^{2d}\right)\;,
$$
as $j\to\infty$ by~(\ref{eq:term2-q0is1}).

If $\mathcal A_2=\emptyset$ we conclude that~(\ref{eq:term1-q0is1}) holds.
By Lemma~\ref{lem:alpha12}, we note that $\mathcal A_2=\emptyset$ if and only
if $d\leq 1/4$ and $I_0$ defined by~(\ref{eq:defJ}) is an empty set.
Hence, from now on, we assume that  $\mathcal A_2\neq\emptyset$, that is, either
$d\leq 1/4$ and $I_0\neq\emptyset$, or $d>1/4$. It only remains to show that, under
these conditions,
~(\ref{eq:reduction-condition}) implies
\begin{equation}
  \label{eq:A2-q0is1}
\lim_{j\to\infty} n_j^{1/2}\gamma_j^{-2d}\sup_{(q,q',p)\in \mathcal A_2} n_j^{-\alpha(q,q',p)}\gamma_j^{\beta'(q,q',p)}=0\;.
\end{equation}
To compute the sup, we first optimize on $p$, then on $q'$, and finally on $q$.

\medskip
\noindent\textbf{Optimization on $p$.} By
Lemma~\ref{lem:alphabeta-incdec}, if $(q,q',p)\in\mathcal A_{2}$ for a given
$(q,q')$, then $\alpha(q,q',p)$ is minimal and $\beta'(q,q',p)$ is maximal for
the largest possible $p$, which corresponds to $p=q-1$ if $q'=q$ and to $p=q$
if $q'>q$. For such a $p$, we have, if $q=q'$,
$$
\alpha(q,q',p)=\alpha(q,q,q-1)=\min(1-2d,1/2)\;,
$$
and if $q'>q$,
$$
\alpha(q,q',p)=\alpha(q,q',q)=1/2-\delta_+(q'-q)\;.
$$
Since being in $\mathcal A_2$ implies
$\alpha(q,q',p)<1/2$, we must have $1-2d<1/2$ (that is $d>1/4$) if $q=q'$ and
$\delta(q'-q)>0$ if $q'>q$.
To separate the cases $q=q'$ and $q\neq q'$, we define
$$
\mathcal A_{2,1}=
\begin{cases}
  \emptyset &\text{ if $d\leq1/4$,}\\
\{(q,q,q-1)~:~q\geq2,\,c_q\neq0\}&\text{ if $d>1/4$.}
\end{cases}
$$
and
$$
\mathcal A_{2,2}=\{(q,q',q)~:~q'>q\geq1,\,c_q c_{q'}\neq0,\,\delta(q'-q)>0\}\;.
$$
Note that in $\mathcal A_{2,1}$ we set $q\geq 2$ to avoid
$(q,q,q-1)=(1,1,0)$. Recall that the indices of the non-zero coefficients $c_q$
are labeled as $q_{\ell}$, see~(\ref{eq:qell-def}).  Then
$$
\{(q,q,q-1)~:~q\geq2,\,c_q\neq0\}=\{(q_{\ell},q_{\ell},q_{\ell}-1)~:~\ell\in\ellset,\,\ell\geq1\}\;,
$$
and, similarly,
$$
\mathcal A_{2,2}=\{(q_{\ell},q_{\ell'},q_{\ell})~:~\ell,\ell'\in\ellset,\,0\leq\ell<\ell',\,\delta(q_{\ell'}-q_{\ell})>0\}\;.
$$
Defining
\begin{equation}\label{e:Ajdef}
A_j:=\sup_{(\ell,\ell')} n_j^{-1/2+\delta(q_{\ell'}-q_{\ell})}\gamma_j^{\beta'(q_{\ell},q_{\ell'},q_{\ell})}\;,
\end{equation}
where the $\sup_{\ell,\ell'}$ is taken over $(\ell,\ell')\in\ellset^2$ such that
$\ell<\ell'$ and $\delta(q_{\ell'}-q_{\ell})>0$, and
\begin{equation}\label{e:Bjdef}
B_j:=\sup_{\ell} \gamma_j^{\beta'(q_{\ell},q_{\ell},q_{\ell}-1)} \;,
\end{equation}
where the $\sup_{\ell}$ is taken over all
$\ell\in\ellset$ such that $\ell\geq1$, we thus obtain the two following assertions.
\begin{itemize}
\item If $d\leq1/4$, the sup over $\mathcal A_2$ can be restricted to $\mathcal
  A_{22}$. This gives
\begin{equation}\label{e:sup1}
\sup_{(q,q',p)\in \mathcal A_2}
  n_j^{-\alpha(q,q',p)}\gamma_j^{\beta'(q,q',p)}= A_j\;,
\end{equation}
\item If $d>1/4$, the sup over $\mathcal A_2$ has to be performed over  $\mathcal
  A_{21}$ and $\mathcal
  A_{22}$. This gives
\begin{equation}\label{e:sup2}
\sup_{(q,q',p)\in \mathcal A_2}
  n_j^{-\alpha(q,q',p)}\gamma_j^{\beta'(q,q',p)}=
\max\left(n_j^{-1+2d}B_j,A_j\right).
\end{equation}
\end{itemize}
\noindent\textbf{Optimization on $q'$.} We only need to consider $A_j$ since
$B_j$ corresponds to $q'=q$. For $A_j$, optimizing on $q'$ means optimizing on
$\ell'$ in the sup of~(\ref{e:sup1}).
We know from  Lemma~\ref{lem:alphabeta-incdec} that, for each
$\ell$, $\alpha(q_{\ell},q_{\ell'},q_{\ell})$ is non-decreasing and
$\beta'(q_{\ell},q_{\ell'},q_{\ell})$ is non-increasing as $\ell'$ increases,
hence the  $\sup_{\ell,\ell'}$ is achieved when $\ell'=\ell+1$ and thus
$\alpha(q_{\ell},q_{\ell'},q_{\ell})<1/2$ implies
\begin{equation}\label{e:AjJd}
A_j=
\begin{cases}
\displaystyle\sup_{\ell\in J_d}
n_j^{-1/2+\delta(q_{\ell+1}-q_{\ell})}\gamma_j^{\beta'(q_{\ell},q_{\ell+1},q_{\ell})}
&\text{ if $J_d\neq\emptyset$,}\\
0&\text{ otherwise}\;,
\end{cases}
\end{equation}
where $J_d$ is defined in~(\ref{e:Jdset}). When $J_d=\emptyset$, the sup
in~(\ref{e:Ajdef}) is taken over the empty set.  We use the convention $\sup_{\emptyset}(\dots)=0$.

\medskip
\noindent\textbf{Optimization on $q$.} We deal separately with the cases
\begin{enumerate}[label=(\alph*)]
\item\label{item-d14a} $d\leq 1/4$.
\item\label{item-d14b} $d>1/4$.
\end{enumerate}
The case~\ref{item-d14a} is the simplest since in~(\ref{e:sup1}), $B_j$ does not appear.
Recall also that we have $I_0\neq\emptyset$ in this case since we assumed $\mathcal
A_2\neq\emptyset$. The optimization on $q$ here amounts to optimize $A_j$ on
$\ell$ in~(\ref{e:AjJd}) in the case $J_d\neq\emptyset$.
Observe that when  $d\leq 1/4$, $\delta(r)=0$ for all
$r\geq2$, see~(\ref{e:ldparamq}). Thus the condition
$\delta(q_{\ell+1}-q_{\ell})>0$ on $\ell\in J_d$ is equivalent to
$q_{\ell+1}=q_{\ell}+1$, in which case
$\delta(q_{\ell+1}-q_{\ell})=\delta(1)=d>0$ and $J_d=I_0$. Hence,
$$
A_j=\sup_{\ell\in I_0}
n_j^{-1/2+d}\gamma_j^{\beta'(q_{\ell},q_{\ell}+1,q_{\ell})}
=\sup_{\ell\in I_0}
n_j^{-1/2+d}\gamma_j^{2\delta_+(q_{\ell})+d-1/2}\;,
$$
where we used that
$\beta'(q_{\ell},q_{\ell}+1,q_{\ell})=\max(2\delta_+(q_{\ell})+d-1/2,-1/2)=2\delta_+(q_{\ell})+d-1/2$,
see~(\ref{e:betap}). This sup is achieved for the smallest $\ell$ since
$\delta_+$ is non-increasing. Recall that the smallest $\ell$ in $I_0$ is denoted
by $\ell_0$ in~(\ref{e:l0}), thus,
$$
A_j=n_j^{-1/2+d}\gamma_j^{2\delta_+(q_{\ell_0})+d-1/2}\;.
$$
Now, we note that in case~\ref{item-d14a} with $q_0=1$, $\nu_c$ in Definition~\ref{def:nuc} takes value
$$
\nu_c=\frac{d+1/2-2\delta_+(q_{\ell_0})}{d}\;.
$$
Hence Condition~(\ref{eq:reduction-condition}) implies
\[
n_j^{1/2}\gamma_j^{-2d} A_j=n_j^{d}\gamma_j^{2\delta_+(q_{\ell_0})-d-1/2}=o(1)\;.
\]
With~(\ref{e:sup1}), we obtain~(\ref{eq:A2-q0is1}) and case~\ref{item-d14a}
is complete.

We now turn to the case~\ref{item-d14b}, that is, we assume now that
$d> 1/4$ and show that~(\ref{eq:A2-q0is1}) holds under
Condition~(\ref{eq:reduction-condition}). Optimizing  $B_j$ on $q$ amounts to optimizing the sup in~(\ref{e:Bjdef}) on
$\ell\in\ellset$ with $\ell\geq1$.
Note that
$\beta'(q_{\ell},q_{\ell},q_{\ell}-1)=\max(2\delta_+(q_{\ell}-1)+2d-1,-1/2)=2\delta_+(q_{\ell}-1)+2d-1$
which is non-increasing as $\ell$ increases. Hence the  sup in~(\ref{e:Bjdef})  is
achieved for $\ell=1$ and thus
\begin{equation}\label{e:Bjsimple}
B_j =\gamma_j^{2\delta_+(q_1-1)+2d-1} \;.
\end{equation}
Note that in this case $\nu_c$ in Definition~\ref{def:nuc} takes value
\begin{equation}
  \label{eq:nuc-dinf1quart}
\nu_c=
\begin{cases}
\frac{1-2\delta_+(q_1-1)}{2d-1/2} &\text{ if $J_d=\emptyset$} \;,\\
\min\left(\frac{1-2\delta_+(q_1-1)}{2d-1/2},\displaystyle\min_{r\in\Irset}
\left(\frac{2d+1/2-2\delta_+(q_{\ell_r})-\delta(r+1)}{\delta(r+1)}\right)\right)
  &\text{ if $J_d\neq\emptyset$.}
\end{cases}
\end{equation}
In both cases, we have  $\nu_c\leq(1-2\delta_+(q_1-1))/(2d-1/2)$,
and thus Condition~(\ref{eq:reduction-condition}) implies, as $j\to\infty$,
$$
\gamma_j^{-2d}n_j^{-1/2+2d}B_j=n_j^{-1/2+2d}\gamma_j^{2\delta_+(q_1-1)-1} = o(1)\;.
$$
If $J_d=\emptyset$ so that $A_j=0$ in~(\ref{e:AjJd}), we thus obtain
with~(\ref{e:sup2}) that~(\ref{eq:reduction-condition})
implies~(\ref{eq:A2-q0is1}).  Similarly, if $J_d\neq\emptyset$, which we now
assume, it only remains to prove that~(\ref{eq:reduction-condition}) implies
\begin{equation}
  \label{eq:AjremainsJdneqempty}
 \lim_{j\to\infty} n_j^{1/2}\gamma_j^{-2d}A_j  =0\;.
\end{equation}
Using~(\ref{e:AjJd}) and that
$\beta'(q_{\ell},q_{\ell+1},q_{\ell})=\max(2\delta_+(q_{\ell})+1/2+\delta(q_{\ell+1}-q_{\ell})-1,1/2)=2\delta_+(q_{\ell})+\delta(q_{\ell+1}-q_{\ell})-1/2$,
we have
\begin{equation}
  \label{eq:Aj-Jd}
A_j=\sup_{\ell\in J_d}
n_j^{-1/2+\delta(q_{\ell+1}-q_{\ell})}\gamma_j^{2\delta_+(q_{\ell})+\delta(q_{\ell+1}-q_{\ell})-1/2} \;.
\end{equation}
Optimizing on $q$ here means optimizing this sup on $\ell\in J_d$.
To do so, we partition $J_d$ as in~(\ref{eq:Jd-decomp}) and,
by the definition of $I_r$ in~(\ref{eq:defIrset}), we have, for all $\ell\in I_{r}$,
$$
n_j^{-1/2+\delta(q_{\ell+1}-q_{\ell})}\gamma_j^{2\delta_+(q_{\ell})+\delta(q_{\ell+1}-q_{\ell})-1/2}
=n_j^{-1/2+\delta(r+1)}\gamma_j^{2\delta_+(q_{\ell})+\delta(r+1)-1/2}\;.
$$
Since $\delta$ is non-increasing, we get with Definition~(\ref{eq:def-ellr}) that, for all $r\in\Irset$,
$$
\sup_{\ell\in I_r}
n_j^{-1/2+\delta(q_{\ell+1}-q_{\ell})}\gamma_j^{2\delta_+(q_{\ell})+\delta(q_{\ell+1}-q_{\ell})-1/2}
= n_j^{-1/2+\delta(r+1)}\gamma_j^{2\delta_+(q_{\ell_{r}})+\delta(r+1)-1/2}\;.
$$
Hence, by~(\ref{eq:Aj-Jd}) and~(\ref{eq:Jd-decomp}), we get
that
$$
n_j^{1/2}\gamma_j^{-2d}A_j=\max_{r\in\Irset}
n_j^{\delta(r+1)}\gamma_j^{2\delta_+(q_{\ell_{r}})+\delta(r+1)-1/2-2d} \;.
$$
Now, since we are in the case $J_d\neq\emptyset$, $\nu_c$
in~(\ref{eq:nuc-dinf1quart}) satisfies
$\nu_c\leq(2d+1/2-2\delta_+(q_{\ell_r})-\delta(r+1))/(\delta(r+1))$ for all
$r\in\Irset$ and recalling that $\Irset$ is a finite set
(see~(\ref{eq:Irsetisfinite})), we see that~(\ref{eq:reduction-condition})
implies~(\ref{eq:AjremainsJdneqempty}). The proof of the case $q_0=1$ is
concluded.

\subsubsection{Proof of Theorem~\ref{th:main2:paper4} in the case $q_0\geq 2$}\label{s:proof:th:case2}
In this case, we need to show that~(\ref{eq:reduction-condition})
implies~(\ref{eq:remains-q0geq2}).

Recall that in Assumptions \textbf{A} include
Condition~(\ref{e:longmemorycondition}) and thus $q_0\geq2$ implies $d>1/4$.
Hence we have $1-2d<1/2$ and since moreover for all $q'\geq q\geq q_0$ and $0\leq
p\leq q$, we have
$\beta(q,p)+\beta(q',p)\leq 2\delta_+(q_0)=2\delta(q_0)$, we obtain that
\[
\lim_{j\to\infty}n_j^{1-2d}\gamma_j^{-2\delta(q_0)}\sup_{(q,q',p)\in\mathcal A_0}
n_j^{-1/2}\gamma_j^{\beta(q,p)+\beta(q',p)}= 0\;.
\]
This correspond to the second term between brackets
in~(\ref{eq:remains-q0geq2}) and we thus only need to prove that
~(\ref{eq:reduction-condition}) implies
\begin{equation}\label{eq:q0neq1-toshow}
\lim_{j\to\infty}n_j^{1-2d}\gamma_j^{-2\delta(q_0)}\sup_{(q,q',p)\in\mathcal A_0}
n_j^{-\alpha(q,q',p)}\gamma_j^{\beta'(q,q',p)}= 0\;.
\end{equation}
Let us partition $\mathcal{A}_0$ into $\mathcal{A}_0=\cup_{i=1}^5\mathcal{A}_i$, where
\begin{equation*}
\begin{array}{lll}
\mathcal A_1&=&\{(q,q',p)\in\mathcal{A}_0,\,q=q'=q_0\},\\
\mathcal A_2&=&\{(q,q,p)\in\mathcal{A}_0,\,q=q'>q_0\},\\
\mathcal A_3&=&\{(q,q',p)\in\mathcal{A}_0,\, q'\geq q+2\},\\
\mathcal A_4&=&\{(q,q',p)\in\mathcal{A}_0,\, q'= q+1,\,p\leq q-1\},\\
\mathcal A_5&=&\{(q,q',p)\in\mathcal{A}_0,\,q=p,\,q'=q+1\}\;.
\end{array}
\end{equation*}
We shall prove that for $i=1,2,3,4$,
\begin{equation}
  \label{eq:q0neq1-toshow-14}
\lim_{j\to\infty}\; n_j^{1-2d}\gamma_j^{-2\delta(q_0)}\sup_{(q,q',p)\in\mathcal A_i}
n_j^{-\alpha(q,q',p)}\gamma_j^{\beta(q,p)+\beta(q',p)}=0 \;,
\end{equation}
and that, when $I_0$ defined as in~(\ref{eq:defJ}) is not empty,~(\ref{eq:reduction-condition}) implies
\begin{equation}
  \label{eq:q0neq1-toshow5}
\lim_{j\to\infty}\; n_j^{1-2d}\gamma_j^{-2\delta(q_0)}\sup_{\ell\in I_0}
n_j^{-\alpha(q_{\ell},q_{\ell}+1,q_{\ell})}\gamma_j^{\beta'(q_{\ell},q_{\ell}+1,q_{\ell})}=0 \;.
\end{equation}
Since $\beta'(q,q',p)\leq \beta(q,p)+\beta(q',p)$ and
$\mathcal A_5=\{(q_{\ell},q_{\ell}+1,q_{\ell})~:~\ell\in I_0\}$, we indeed
have that~(\ref{eq:q0neq1-toshow-14}) and~(\ref{eq:q0neq1-toshow5})
imply~(\ref{eq:q0neq1-toshow}) and the proof will be concluded.

The limit~(\ref{eq:q0neq1-toshow-14}) can be deduced for  $i=1,\cdots,4$ from Lemma~8.3
of~\cite{clausel-roueff-taqqu-tudor-2013}. More precisely this lemma implies
the following facts (recall that $d>1/4$).
\begin{enumerate}
\item For all $p=0,\dots,q_0-2$, we have
$\alpha(q_0,q_0,p)>1-2d$ and $2\beta(q_0,p)\leq
2\delta(q_0)$, which implies~(\ref{eq:q0neq1-toshow-14}) for $i=1$ since $(q_0,q_0,q_0-1)$ is excluded from $\mathcal{A}_0$.
\item For all $q\geq q_0+1$ and $p=0,\dots,q-1$, we have
  $\alpha(q,q,p)\geq 1-2d$ and $2\beta(q,p)\leq2\delta_+(q_0+1)<2\delta(q_0)$, which implies~(\ref{eq:q0neq1-toshow-14}) for $i=2$.
\item For $q\geq q_0$, $q'\geq q+2$ and $p=0,\dots,q$, we have $\alpha(q,q',p)\geq
  1-2d$ and  $\beta(q,p)+\beta(q',p)\leq\delta_+(q_0)+\delta_+(q_0+2)<2\delta(q_0)$,  which implies~(\ref{eq:q0neq1-toshow-14}) for $i=3$.
\item For $q\geq q_0$ and
  $p=0,\dots,q-1$, we have $\alpha(q,q+1,p)\geq\min(3/2(1-2d),1/2)>1-2d$ and
  $\beta(q,p)+\beta(q+1,p)\leq\delta_+(q_0+1)+\delta(q_0)<2\delta(q_0)$,  which implies~(\ref{eq:q0neq1-toshow-14}) for $i=4$.
\end{enumerate}
Hence we obtain that~(\ref{eq:q0neq1-toshow-14}) is valid for $i=1,\cdots,4$.
If $I_0$ is empty, the proof is concluded.
We now assume that $I_0$ is not empty, so
that $\ell_0$ is finite, and it only remains to show
that Condition~(\ref{eq:reduction-condition}) implies~(\ref{eq:q0neq1-toshow5}).
Observe that, for any $q\geq q_0$, we have $\alpha(q,q+1,q)=1/2-d$ and
$\beta'(q,q+1,q)=2\delta_+(q)+d-1/2$ is non--increasing as $q$ increases. Hence
over $\ell\in I_0$, $\alpha(q_{\ell},q_{\ell}+1,q_{\ell})$ is constant and
$\beta'(q_{\ell},q_{\ell}+1,q_{\ell})$ is maximal at $\ell=\ell_0$,
where it takes value $2\delta_+(q_{\ell_0})+d-1/2$. We conclude that
$$
n_j^{1-2d}\gamma_j^{-2\delta(q_0)}\sup_{\ell\in I_0}
n_j^{-\alpha(q_{\ell},q_{\ell}+1,q_{\ell})}\gamma_j^{\beta'(q_{\ell},q_{\ell}+1,q_{\ell})}
= n_j^{1/2-d}\gamma_j^{d-1/2-2(\delta(q_0)-\delta_+(q_{\ell_0}))}\;.
$$
Note that in this case $\nu_c$ in Definition~\ref{def:nuc} takes value
\[
\nu_c= 1+\frac{4(\delta(q_0)-\delta_+(q_{\ell_0}))}{1-2d} \;.
\]
Thus Condition~(\ref{eq:reduction-condition}) implies~(\ref{eq:q0neq1-toshow5})
and the proof is finished.

\subsection{Proof of Theorem~\ref{thm:test}}\label{sec:proof-hyp-test}

The fact that the test  $\delta_s$ is consistent follows directly from the
consistency statement in Theorem~\ref{thm:cons-wavelet-log-regression} and the fact that $(u_N)$ is diverging.

To show that the test $\delta_s$ has asymptotic confidence level $\alpha$, it suffices to
show that when $d_0=d_0^*$ (null hypothesis), we have
\begin{equation}\label{e:had-d0-clt}
u_N(\hat{d}_0-d_0)\overset{(\mathcal{L})}{\rightarrow}U\;.
\end{equation}
We first observe that under the conditions on $j=(j_N)$ of the theorem, the
convergence~(\ref{eq:dtilde-clt}) involving $\widetilde{d}_0$ holds,
see \cite{clausel-roueff-taqqu-tudor-2011b}.

The
computations of Section~5 in \cite{clausel-roueff-taqqu-tudor-2011b} allows us
to specify~(\ref{e:hat-d0-approx}) as
$$
\hat d_{0}-d_0 =
L(2^{-2d_0j}\overline{\mathbf{S}}_{n_j,j})+o_P\left(2^{-2d_0j}\overline{\mathbf{S}}_{n_j,j}\right)+O(2^{-\zeta
  j})\;,
$$
where $L$ is the linear form
\[
L(z_1,\cdots,z_p)=\sum_{i=1}^p w_i 2^{-2d_0 i}z_i\;,
\]
where the weights $(w_i)$ have been defined in Section~\ref{sec:stat-inf-setting}.
The same linearization holds for
$\widetilde{d}_0-d_0$ with $\overline{\mathbf{S}}_{n_j,j}$ replaced by $\mathbf{S}_{n_j,j}^{(q_0,q_0,q_0-1)}$, so by subtracting, we get
\begin{equation}\label{e:hat-d0-approx-precised}
  \hat d_{0}-d_0 =
  \widetilde{d}_0-d_0+2^{-2d_0j}\left[O_P\left(\left|\overline{\mathbf{S}}_{n_j,j}-\mathbf{S}_{n_j,j}^{(q_0,q_0,q_0-1)}\right|\right)+o_P\left(\left|\mathbf{S}_{n_j,j}^{(q_0,q_0,q_0-1)}\right|\right)\right]+O(2^{-\zeta
    j})\;.
\end{equation}
Using~(\ref{eq:reduction-test}), which corresponds
to~(\ref{eq:reduction-condition}) under $H_0$, we can
apply Theorem~\ref{th:main2:paper4} so that
$$
\overline{\mathbf{S}}_{n_j,j}-\mathbf{S}_{n_j,j}^{(q_0,q_0,q_0-1)}=o_P\left(\left\|\mathbf{S}_{n_j,j}^{(q_0,q_0,q_0-1)}\right\|_2\right)\;.
$$
With~(\ref{e:hat-d0-approx-precised}), we get
$$
  \hat d_{0}-d_0 =
  \widetilde{d}_0-d_0+o_P\left(2^{-2d_0j}\left\|\mathbf{S}_{n_j,j}^{(q_0,q_0,q_0-1)}\right\|_2\right)+O(2^{-\zeta
    j})\;.
$$
Since $u_N(\widetilde{d}_0-d_0)$ converges in distribution, it remains to check
that $u_N\left\|\mathbf{S}_{n_j,j}^{(q_0,q_0,q_0-1)}\right\|_2=O(1)$ and
$u_N2^{-\zeta j}=o(1)$. By the definition of $u_N$
in~(\ref{eq:dtilde-clt}) and since $\gamma_j=2^j$, $d_0=K+\delta(q_0)$ and
$\delta(1)=d$, the asymptotic equivalences~(\ref{eq:q0eq1-rate-leading})
and~(\ref{eq:q0neq1-rate-leading}) in Proposition~\ref{prop:sharp-rates-mainterm} can be written as
$$
\|S_{n_j,j}^{(q_0,q_0,q_0-1)}\|_2\sim C\;u_N^{-1}2^{2d_0j}\;.
$$
The bound $u_n\left\|\mathbf{S}_{n_j,j}^{(q_0,q_0,q_0-1)}\right\|_2=O(1)$
follows under $H_0$. Finally the bound $u_N2^{-\zeta j}=o(1)$ follows from the
bias negligibility condition~(\ref{eq:bias-negligibility-condition}). Hence we
get~(\ref{e:had-d0-clt}), which concludes the proof.
\section{Technical lemmas}\label{s:lemmas}
The next lemma give an explicit expression of the spectral density of $H_q(X)$ for $q<1/(1-2d)$ and is a refined version of Lemma~4.1 in~\cite{clausel-roueff-taqqu-tudor-2011a}. It is used in the proof of Theorem~\ref{thm:spec-density}.
\begin{lemma}\label{lem:spec-density}
Let $q$ be a positive integer greater than $2$. The spectral density of $\{H_q(X_\ell)\}_{\ell\in\mathbb{Z}}$ is
\begin{equation}\label{e:spec-dens-Hq}
f_{H_q}:=q!(f\star\cdots \star f)\;,
\end{equation}
where $f$ denotes the spectral density of $X=\{X_\ell\}_{\ell\in\mathbb{Z}}$. Moreover if in addition $q<1/(1-2d)$ the function $f^*_{H_q}$ in
\begin{equation}\label{e:fstar-Hq}
f_{H_q}(\lambda)=|1-\rme^{-\rmi\lambda}|^{-2\delta(q)}f^*_{H_q}(\lambda)\;.
\end{equation}
is bounded on $\lambda\in (-\pi,\pi)$ and for any $\zeta\in (0,2\delta(q))$ such that $\zeta\leq \beta$, where $\beta$ has been defined in~(\ref{e:beta-def}), one has
\begin{equation}\label{e:fstar-Hq-lip}
|f^*_{H_q}(\lambda)-f^*_{H_q}(0)|\leq L|\lambda|^\zeta\;,
\end{equation}
for some $L>0$.
\end{lemma}
\begin{proof}
The explicit expression~(\ref{e:spec-dens-Hq}) of $f_{H_q}$ has already been given in Lemma~4.1 in~\cite{clausel-roueff-taqqu-tudor-2011a}. Moreover in the same lemma, we also already showed that $f^*_{H_q}$ defined by~(\ref{e:fstar-Hq}) is a bounded function. We then only need to prove that~(\ref{e:fstar-Hq-lip}) holds for some $L>0$. We prove the result by induction on $q$.

Assume first that $q=2$. By assumption on $f^*$ and definition of $\beta$, we know that for some $C>0$ and any $\zeta\leq\beta$
\begin{equation}\label{e:f-star}
|f^*(\lambda)-f^*(0)|\leq C|\lambda|^\zeta\;.
\end{equation}
Since $f_{H_2}=2f*f$, we then apply the second part of Lemma~8.2 of~\cite{clausel-roueff-taqqu-tudor-2011a}, with $\beta_1=\beta_2=2d$, $g_1^*=g_2^*=f^*$ (using the notations of that lemma). We see that Condition (66) of Lemma~8.2 of~\cite{clausel-roueff-taqqu-tudor-2011a} is satisfied provided that $\zeta\leq \beta$ and $\zeta<\beta_1+\beta_2-1=2d+2d-1=2\delta(2)$ (which are necessary conditions of the lemma). Hence for some $L>0$, one has
\[
|f^*_{H_2}(\lambda)-f^*_{H_2}(0)|\leq L|\lambda|^\zeta\;.
\]

If we now assume that $q>2$, we can also apply the second part of Lemma~8.2 of~\cite{clausel-roueff-taqqu-tudor-2011a}, with $\beta_1=2\delta(q-1)$, $\beta_2=2d$, $g_1^*=f^*_{H_{q-1}}$ and $g_2^*=f^*$ which allows us to proceed by induction.
\end{proof}
Lemmas~\ref{lem:peq0} to~\ref{lem:alpha12}
are used in the proof of Theorem~\ref{th:main2:paper4}.
\begin{lemma}\label{lem:peq0}
Let $\delta^+$ be the exponent defined in~(\ref{e:ldparamq}). One has
\begin{equation}\label{e:UB:beta:peq0}
\sup\left\{\delta^+(q)+\delta^+(q')~:~1\leq q\leq q',\,(q,q')\neq(1,1)\right\}<2d\;.
\end{equation}
\end{lemma}
\begin{proof}
For any $(q,q')$ in the considered set, one has $q\geq 1$ and
$q'\geq2$. Since $\delta_+$ is non-increasing, we get
$\delta^+(q)+\delta^+(q')\leq \delta^+(1)+\delta^+(2)=d+(2d-1/2)_+<2d$ since $d<1/2$. Lemma~\ref{lem:peq0} follows.
\end{proof}
\begin{lemma}\label{lem:alphabeta-incdec}
  Let $\alpha(q,q',p)$ and $\beta'(q,q',p)$ be the exponents defined
  in~(\ref{e:alpha}) and (\ref{e:betap}) respectively, for $0\leq p\leq q\leq
  q'$. Then the following facts hold~:
  \begin{enumerate}[label=(\roman*)]
  \item\label{item:alpha-inc} $\alpha(q,q',p)$ is non-decreasing as $q$ or $q'$ increases and is
  non-increasing as $p$ increases.
\item\label{item:beta-decc} $\beta'(q,q',p)$ is non-increasing as $q$
  or $q'$ increases.
\item\label{item:beta-inc} On the set $\{(q,q',p)~:~0\leq p\leq q\leq
  q',\,\alpha(q,q',p)<1/2\}$, $\beta'(q,q',p)$ is non-decreasing as $p$
  increases.
\end{enumerate}
\end{lemma}
\begin{proof}
  The facts~\ref{item:alpha-inc} and~\ref{item:beta-decc} directly follow by observing that $\delta_+$ is a non-increasing
  function. Now suppose that $\alpha(q,q',p)<1/2$. It follows that $p\neq0$ and
  $\delta_+(q-p)>0$ and $\delta_+(q'-p)>0$ since otherwise
  $\delta_+(q-p)+\delta_+(q'-p)\leq1/2$ which implies $\alpha(q,q',p)=1/2$ in
  the case $p\neq0$. Now, when  $\delta_+(q-p)>0$ and $\delta_+(q'-p)>0$, we
  have  in the definition of $\beta'$ that
$$
\beta'(q,q',p)=
\begin{cases}
\max(\delta(q-p)+\delta(q'-p)-1,-1/2)&\text{ if $\delta_+(p)=0$}\\
\max(\delta(q)+\delta(q'),-1/2)&\text{ if $\delta_+(p)>0$.}
\end{cases}
$$
The second line comes from the fact that
$2\delta(p)+\delta(q-p)+\delta(q'-p)-1=\delta(q)+\delta(q')$. Now it is clear
that $\beta'(q,q',p)$ is non-decreasing as $p$ increases.
\end{proof}
\begin{lemma}\label{lem:alpha12}
Let $\alpha(q,q',p)$ be the exponent defined in~(\ref{e:alpha}) for $0\leq p\leq q\leq q'$. Then we have $\alpha(q,q',p)\in(0,1/2]$ and the three
following assertions hold~:
\begin{enumerate}[label=(\roman*)]
\item\label{item:alpha120} For any $q\geq1$, $\alpha(q,q+1,q)=1/2-d<1/2$.
\item\label{item:alpha121} If $d\leq1/4$, then for all  $1\leq q\leq
    q'$ and $0\leq p\leq \min(q,q'-1)$ such that $q'\neq q+1$, we have $\alpha(q,q',p)=1/2$.
\item\label{item:alpha122} If $d>1/4$, then for all $q\geq2$, $\alpha(q,q,q-1)=1-2d<1/2$.
\end{enumerate}
\end{lemma}
\begin{proof}
  Assertion~\ref{item:alpha120} follows by applying~(\ref{e:alpha}), since
  $\delta_+(0)=1/2$ and  $\delta_+(1)=d$.

  Suppose that $d\leq1/4$. If $p=0$,
  $\alpha(q,q',p)=1/2$ for any $q,q'$ by~(\ref{e:alpha}). Let now
  $p\geq1$. Let $(q,q')$ be such that   $1\leq q\leq
    q'$, $0\leq p\leq \min(q,q'-1)$ and $q'\neq q+1$. Then either $q=q'\geq
    p+1$ (first case) or $q'\geq q+2$ and $q\geq p$ (second case).
Then by Lemma~\ref{lem:alphabeta-incdec}(\ref{item:alpha-inc}), we
  have in the first case
$$
\alpha(q,q,p)\geq \alpha(p+1,p+1,p)=\min(1-2d,1/2)=1/2\;,
$$
since $d\leq1/4$. In the second case,
Lemma~\ref{lem:alphabeta-incdec}~\ref{item:alpha-inc} implies
$$
\alpha(q,q',p)\geq \alpha(p,p+2,p)=\min(1/2-\delta_+(2),1/2)=1/2\;,
$$
since $d\leq1/4$ implies $\delta_+(2)=0$. This proves
Assertion~\ref{item:alpha121}.

To obtain Assertion~\ref{item:alpha122}, we remark that
$$
\alpha(q,q,q-1)=\min(1-2d,1/2)=1-2d<1/2\;,
$$
since $d>1/4$.
\end{proof}

\begin{lemma}
  \label{lem:nuc-wrt-d}
  Consider a sequence $\{q_{\ell},\,\ell\in\ellset\}$ with $\ellset$ a set of
  consecutive integers starting at $0$. Let $\nu_c(d)$ be as in
  Definition~\ref{def:nuc} for all $d\in (1/2(1-1/q_0),1/2)$, so
  that~(\ref{e:longmemorycondition}) holds. Then the following assertions hold~:
\begin{enumerate}[label=(\roman*)]
\item\label{item:nuc-wrt-d-1} If $q_0=1$, $\nu_c(d)$ is non-increasing as $d$ increases.
\item\label{item:nuc-wrt-d-2} If $q_0\geq 2$, $\nu_c(d)$ is non-decreasing as $d$ increases.
\end{enumerate}
\end{lemma}
\begin{proof}
  We first consider the case $q_0\geq 2$. In this case, either $I_0=\emptyset$ and
  $\nu_c(d)=\infty$, or $I_0\neq\emptyset$ and $\nu_c$ is a continuous function
  taking values
$$
\nu_c(d) =
\begin{cases}
1+\frac{4(\delta(q_0)-\delta(q_{\ell_0}))}{1-2d} =  1+2(q_{\ell_0}-q_0)&\text{
  if $\delta(q_{\ell_0})>0$,}\\
1+\frac{4\delta(q_0)}{1-2d} = 1-2q_0+2/(1-2d)&\text{ otherwise.}
\end{cases}
$$
Hence we obtain~\ref{item:nuc-wrt-d-2}.

We now consider the case $q_0=1$. In this case, with the convention
$a/0=\infty$ for $a>0$ the following formula can be applied in all cases~:
$$
\nu_c(d)=\min\left(\frac{1-2\delta_+(q_1-1)}{\delta_+(2)},
\frac{2d+1/2-2\delta_+(q_{\ell_r})-\delta(r+1)}{\delta(r+1)}~:~r\in\Irset\right)
$$
This comes from the fact that if $d\leq1/4$, we have $\delta_+(2)=0$ and
$\Irset\subset\{0\}$ with equality if and only if $I_0\neq\emptyset$.
Let us denote
$$
\tilde\Irset=\left\{1+q_{\ell+1}-q_\ell~:~\ell\in\ellset\right\}\;,
$$
so that $\Irset=\tilde\Irset\cap\{r~:~\delta(r+1)>0\}$. Since
$2d+1/2-2\delta_+(q_{\ell_r})-\delta(r+1)=2(d-\delta_+(q_{\ell_r}))+(1/2-\delta(r+1))>0$
we get using the same convention as above that
$$
\nu_c(d)=\min\left(\frac{1-2\delta_+(q_1-1)}{\delta_+(2)},
\frac{2d+1/2-2\delta_+(q_{\ell_r})-\delta(r+1)}{\delta_+(r+1)}~:~r\in\tilde\Irset\right)\;,
$$
where now the set $\tilde\Irset$ does not depend on $d$.
To prove~\ref{item:nuc-wrt-d-1}, we thus only need to show the following two assertions (setting
$q=q_1-1$ and then $p=r+1$ and $q=q_{\ell_r}$).
\begin{enumerate}[label=(\alph*)]
\item\label{item:nuc-wrt-d-a} For any given positive integer $q$, $(1-2\delta_+(q))/\delta_+(2)$ is non-increasing as $d$ increases,
\item\label{item:nuc-wrt-d-c} For any given positive integers $p$ and $q$,
  $\mu(d):=(2d+1/2-2\delta_+(q)-\delta(p))/{\delta_+(p)}$ is non-increasing
  as $d$ increases.
\end{enumerate}
Assertion~\ref{item:nuc-wrt-d-a} follows from the fact that $\delta(q)$ is
increasing with $d$ for any given $q\geq1$.
Finally, we need to prove Assertion~\ref{item:nuc-wrt-d-c}. Take some
integers $p,q\geq1$ and denote $\mu(d)$ as in~\ref{item:nuc-wrt-d-c}. If
$\delta_+(p)=0$, which is equivalent to $d\leq1/2(1-1/p)$, $\mu(d)=\infty$. Now
$\mu(d)$ is continuous over $d>1/2(1-1/p)$ and takes value
$$
\mu(d)=\min\left(\frac{1/2+2d}{dp+(p-1)/2},\frac{1/2+(q-1)(1-2d)}{dp+(p-1)/2}\right)\;.
$$
Since the two arguments in the min are decreasing functions of $d$ over
$d>1/2(1-1/p)$, we conclude that~\ref{item:nuc-wrt-d-c} holds. The proof of
the lemma is achieved.
\end{proof}

\appendix
\section{Integral representations}\label{s:appendixA}
It is convenient to use an integral representation in the spectral
domain to represent the random processes (see for
example~\cite{major:1984,nualart:2006}). The stationary Gaussian
process $\{X_k,k\in\mathbb{Z}\}$ with spectral
density~(\ref{e:sdf}) can be written as
\begin{equation}\label{e:intrepX}
X_\ell=\int_{-\pi}^{\pi}\rme^{\rmi\lambda
  \ell}f^{1/2}(\lambda)\rmd\widehat{W}(\lambda)=\int_{-\pi}^{\pi}\frac{\rme^{\rmi\lambda
    \ell}f^{*1/2}(\lambda)}{|1-\rme^{-{\rmi}\lambda}|^d}\rmd\widehat{W}(\lambda),\quad\ell\in\mathbb{N}\;.
\end{equation}
This is a special case of
\begin{equation}\label{e:int}
\widehat{I}(g)=\int_{\mathbb{R}}g(x)\rmd\widehat{W}(x),
\end{equation}
where $\widehat{W}(\cdot)$ is a complex--valued Gaussian random
measure satisfying, for any Borel sets $A$ and $B$ in
$\mathbb{R}$, $\mathbb{E}(\widehat{W}(A))=0$,
$\mathbb{E}(\widehat{W}(A)\overline{\widehat{W}(B)})=|A\cap B|$
and
$$
\widehat{W}(A)=\overline{\widehat{W}(-A)}\;.
$$
The integral~(\ref{e:int}) is defined for any function $g\in
L^2(\mathbb{R})$ and one has the isometry
\[
\mathbb{E}(|\widehat{I}(g)|^2)=\int_{\mathbb{R}}|g(x)|^2\rmd x\;.
\]
The integral $\widehat{I}(g)$, moreover, is real--valued if
\[
g(x)=\overline{g(-x)}\;.
\]

We shall also consider multiple It\^{o}--Wiener integrals
\[
\widehat{I}_q(g)=\int^{''}_{\mathbb{R}^q}g(\lambda_1,\cdots,\lambda_q)\rmd
\widehat{W}(\lambda_1)\cdots\rmd\widehat{W}(\lambda_q)
\]
where the double prime indicates that one does not integrate on
hyperdiagonals $\lambda_i=\pm \lambda_j,i\neq j$. The integrals
$\widehat{I}_q(g)$ are handy because we will be able to expand our
non--linear functions $G(X_k)$ introduced in Section~\ref{s:intro}
in multiple integrals of this type.

These multiples integrals are defined for
$g\in\overline{L^2}(\mathbb{R}^{q},\mathbb{C})$, the space of
complex valued functions defined on $\mathbb{R}^{q}$ satisfying
\begin{gather}\label{e:antisym}
g(-x_{1},\cdots,-x_{q})= \overline{g(x_{1},\cdots, x_{q})}\mbox{ for }(x_{1},\cdots, x_{q}) \in \mathbb{R}^q\;,\\
\label{e:fL2} \Vert g\Vert ^{2} _{L^{2}}:= \int_{\mathbb{R}^{q}}
\left| g(x_{1},\cdots, x_{q}) \right| ^{2} \rmd x_{1}\cdots \rmd
x_{q} <\infty\;.
\end{gather}
Hermite polynomials are related to multiple integrals as follows :
if $X=\int_{\mathbb{R}}g(x)\rmd\widehat{W}(x)$ with
$\mathbb{E}(X^2)=\int_{\mathbb{R}}|g(x)|^2\rmd x=1$ and
$g(x)=\overline{g(-x)}$ so that $X$ has unit variance and is
real--valued, then
\begin{equation}\label{e:herm-integ}
H_q(X)=\widehat{I}_q(g^{\otimes
q})=\int_{\mathbb{R}^q}^{''}g(x_1)\cdots g(x_q)\rmd
\widehat{W}(x_1)\cdots\rmd\widehat{W}(x_q)\;.
\end{equation}
\section{The wavelet filters}\label{s:appendixB}
The sequence $\{Y_{t}\}_{t\in\mathbb{Z}}$ can be formally
expressed as
$$
Y_{t}=\Delta^{-K}G(X_{t}), \quad t\in\mathbb{Z}\;.
$$
The study of the asymptotic behavior of the scalogram of $\{Y_{t}\}_{t\in\mathbb{Z}}$ at different scales involve multidimensional wavelets coefficients of
$\{G(X_t)\}_{t\in\mathbb{Z}}$ and of $\{Y_t\}_{t\in\mathbb{Z}}$.
To obtain them, one applies a multidimensional linear filter
$\bh_j(\tau),\tau\in\mathbb{Z}=(h_{j,\ell}(\tau))$, at each scale index $j\geq 0$. We shall
characterize below the multidimensional filters $\bh_j(\tau)$ by their discrete
Fourier transform~:
\begin{equation}\label{e:dF}
\widehat{\bh}_j(\lambda)=\sum_{\tau\in\mathbb{Z}}\bh_j(\tau)
\rme^{-\rmi \lambda\tau },\,\lambda\in [-\pi,\pi]\;,\quad
\bh_j(\tau)=\frac{1}{2\pi}\int_{-\pi}^{\pi}\widehat{\bh}_j(\lambda)\rme^{\rmi\lambda\tau}\rmd\lambda,\tau\in\mathbb{Z}\;.
\end{equation}
The resulting wavelet coefficients $\bW_{j,k}$, where $j$ is the
scale index and $k$ the location are defined as
\begin{equation}\label{e:W1}
\bW_{j,k}=\sum_{t\in\mathbb{Z}}\bh_j(\gamma_j
k-t)Y_{t}=\sum_{t\in\mathbb{Z}}\bh_j(\gamma_j
k-t)\Delta^{-K}G(X_{t}),\,j\geq 0, k\in\mathbb{Z},
\end{equation}
where $\gamma_j\uparrow \infty$ as $j\uparrow \infty$ is a
sequence of non--negative scale factors applied at scale index $j$, for
example $\gamma_j=2^j$.  We do not assume that the wavelet
coefficients are orthogonal nor that they are generated by a
multiresolution analysis. Our assumption on the filters
$\bh_j=(h_{j,\ell})$ are as follows~:
\begin{enumerate}[label=(W-\arabic*)]
\item\label{ass:w-ap} \underline{Finite support}: For each $\ell$ and $j$,
  $\{h_{j,\ell}(\tau)\}_{\tau\in\mathbb{Z}}$ has finite support. Further there exists some $A>0$ such that for any $j$ and any $\ell$ one has
  \begin{equation}\label{e:size-support}
  \mathrm{supp}(h_{j,\ell})\subset \gamma_j [-A,A]\;.
  \end{equation}
\item\label{ass:w-bp} \underline{Uniform smoothness}: There exists
$M\geq K$, $\alpha>1$
  and $C>0$ such that for all $j\geq0$ and $\lambda\in [-\pi,\pi]$,
    \begin{equation}\label{e:majoHj}
    |\widehat{\bh}_j(\lambda)|\leq \frac{C\gamma_j^{1/2}|\gamma_j\lambda|^M}{(1+\gamma_j|\lambda|)^{\alpha +M}}\;.
    \end{equation}
By $2\pi$-periodicity of $\widehat{h}_j$ this inequality can be
extended to $\lambda\in\mathbb{R}$ as
\begin{equation}\label{EqMajoHjR}
|\widehat{\bh}_{j}(\lambda)|\leq C
\frac{\gamma_j^{1/2}|\gamma_j\{\lambda\}|^M}
{(1+\gamma_j|\{\lambda\}|)^{\alpha+M}}\;.
\end{equation}
where $\{\lambda\}$ denotes the element of $(-\pi,\pi]$ such that
$\lambda-\{\lambda\}\in2\pi\mathbb{Z}$.
\item\label{ass:w-c}
\underline{Asymptotic behavior}: There exists a sequence of phase functions $\Phi_j :\mathbb{R}\rightarrow (-\pi,\pi]$ and some non identically
  zero function $ \widehat{\bh}_{\infty}$ such that
\begin{equation}\label{EqLimHj}
\lim_{j \to
+\infty}(\gamma_j^{-1/2}\widehat{\bh}_{j}(\gamma_j^{-1}\lambda))=
\widehat{\bh}_{\infty}(\lambda)\;,
\end{equation}
locally uniformly on $\lambda\in\mathbb{R}$.
\end{enumerate}
In~\ref{ass:w-c} {\it locally uniformly} means that for all compact
$K\subset \mathbb{R}$,
\[
\sup_{\lambda\in K}\left|\gamma_j^{-1/2}\widehat{\bh}_j(\gamma_j^{-1}\lambda)\rme^{\rmi \Phi_j(\lambda)}-\widehat{\bh}_{\infty}(\lambda)\right|\to 0\;.
\]
It implies in particular that $\widehat{\bh}_{\infty}$ is continuous over
$\mathbb{R}$.

A more convenient way to express the wavelet coefficients
$\bW_{j,k}$ than in~(\ref{e:W1}) is to incorporate the linear
filter $\Delta^{-K}$ into the filter $\bh_j$ and
denote the resulting filter $\bh_j^{(K)}$. Then
\begin{equation}\label{e:W2}
\bW_{j,k}=\sum_{t\in\mathbb{Z}}\bh_j^{(K)}(\gamma_j k-t)G(X_{t})\;,
\end{equation}
where
\begin{equation}\label{e:HjkVSHj}
\widehat{\bh}_j^{(K)}(\lambda)=(1-\rme^{-{\rmi}\lambda})^{-K}
\widehat{\bh}_j(\lambda)
\end{equation}
is the discrete Fourier transform of $\bh_j^{(K)}$, see
\cite{clausel-roueff-taqqu-tudor-2013} for more details.
\section{The multiscale wavelet inference setting}\label{s:appendixC}
We state here two theorems that are used in Section~\ref{sec:appl-wavel-estim}
to derive statistical properties of the estimator of the memory parameter $d_0$. This parameter is obtained from univariate multiscale wavelet filters $g_j$. Since, Theorem~\ref{th:main:paper4} applies to multivariate filters $\bh_j$ which define the multivariate scalogram $\bS_{n,j}$, we explain in this appendix the connection between these two perspectives.

We first give some details about the definition of the estimator of the memory parameter. We use dyadic scales
here, as in the standard wavelet analysis described in
\cite{moulines:roueff:taqqu:2007:jtsa}, where the univariate wavelet coefficients are
defined as
\begin{equation}
  \label{eq:wav_coeff_def1_dyad}
 W_{j,k}=\sum_{t\in\mathbb{Z}}g_j(2^j k-t)Y_{t}\;,
\end{equation}
which corresponds to~(\ref{e:WC}) with $\gamma_j=2^j$ and with
$(g_j)$ denoting a sequence of filters that
satisfies~\ref{ass:w-ap}--\ref{ass:w-c} with $m=1$.
In the case of a multiresolution analysis, $g_j$ can be
deduced from the associated mirror filters.

The number $n_j$ of wavelet coefficients available at
scale $j$, is related both to the number $N$ of observations $Y_1,\cdots,Y_N$
of the time series $Y$ and to the length $T$ of the support of the wavelet
$\psi$. More precisely, one has
\begin{equation}
  \label{eq:nj_def}
  n_j=[2^{-j}(N-T+1)-T+1]= 2^{-j}N+0(1)\;,
\end{equation}
where $[x]$ denotes the integer part of $x$ for any real $x$.
Details about the above facts can be found in
\cite{moulines:roueff:taqqu:2007:jtsa,roueff-taqqu-2009b}.

The univariate scalogram is an empirical measure of the distribution of
``energy of the
signal'' along scales, based on the $N$ observations $Y_1,\cdots,Y_N$. It is
defined as
\begin{equation}
  \label{eq:scalo_def}
\widehat{\sigma}^2_{j}=\frac1{n_j}\sum_{k=0}^{n_j-1}W_{j,k}^2, \quad j\geq0 \;,
\end{equation}
and is identical to $S_{n_j,j}$ defined in~(\ref{e:defsnj}). The wavelet spectrum is defined as
\begin{equation}
  \label{eq:wav-sp_def}
{\sigma}^2_{j}=\mathbb{E}[\widehat{\sigma}^2_{j}]=\mathbb{E}[W_{j,k}^2]\quad\text{for
  all $k$}\;,
\end{equation}
where the last equality holds for $M\geq K$ since in this case
$\{W_{j,k},k\in\mathbb{Z}\}$ is weakly stationary.

To define our wavelet estimator of the memory parameter $d_0$, we are given some positive weights $w_0,\cdots,w_p$ such that
\[
\sum_{i=0}^p w_i=0\mbox{ and }\sum_{i=0}^p i w_i=\frac{1}{2\log(2)}\;.
\]
We then set
\begin{equation}
\label{e:d0}\hat d_0=\sum_{i=0}^p w_i\log(\widehat{\sigma}_{j+i})\;.
\end{equation}

To derive statistical properties of this estimator, we apply
Theorem~\ref{th:main:paper4} using a sequence of multivariate filters
$(\bh_j)_{j\geq0}$ related to the family of univariate filters $g_j$ in a way indicated below.

We first give an example and consider the case $p=1$. To investigate the asymptotic properties of $\hat d_0$, we then have to study the {\it joint behavior} of $W_{j-u,k}$ for $u=0,1$. Recall that $j-1$ is a finer scale than $j$. Following the framework of
\cite{roueff-taqqu-2009b}, we consider the multivariate coefficients
$\bW_{j,k}=(W_{j,k},\,W_{j-1,2k},\,W_{j-1,2k+1})$, since, in addition to the
wavelet coefficients $W_{j,k}$ at scale $j$, there are twice as many wavelet
coefficients at scale $j-1$, the additional coefficients being $W_{j-1,2k},\,W_{j-1,2k+1}$.  These coefficients
can be viewed in this case as the output of a three-dimensional filter $\bh_j$ defined as
$\bh_j(\tau)=(g_j(\tau),\,g_{j-1}(\tau),\,g_{j-1}(\tau+2^{j-1}))$.
These three entries correspond to $(u,v)$ below equal to $(0,0)$, $(1,0)$ and
$(1,1)$, respectively, in the general case below.

In the general case, each $\bh_j$ is defined as follows.
For all, $j\geq0$, $u\in\{0,\dots,j\}$
and $v\in\{0,\dots,2^u-1\}$, let $\ell=2^u+v$ and define a filter $h_{\ell,j}$
by
\begin{equation}
  \label{eq:multiscale-filters}
  h_{\ell,j}(t)=g_{j-u}(t+2^{j-u} v),\quad t\in\mathbb{Z}\;.
\end{equation}
Applying this definition and~(\ref{eq:wav_coeff_def1_dyad}) with $\gamma_j=2^j$, we
get
$$
W_{j-u,2^uk+v}=\sum_{t\in\mathbb{Z}}h_{\ell,j}(2^j k-t)Y_t\;.
$$
These coefficients are stored in a vector $\bW_{j,k}=[W_{\ell,j,k}]_\ell$, say
of length $m=2^p-1$,
\begin{equation}
  \label{eq:Wljk-multiscale}
W_{\ell,j,k}=W_{j-u,2^uk+v}, \quad\ell=2^u+v=1,2,\dots,m\,,
\end{equation}
which corresponds to the multivariate wavelet coefficient~(\ref{e:Wjkbold}) with
$\bh_{j}(t)$ having components $h_{\ell,j}(t)$, $\ell=1,2,\dots,m$ defined
by~(\ref{eq:multiscale-filters}). This way of proceeding allows us to express
the vector $[\widehat{\sigma}^2_{j-u}-{\sigma}^2_{j-u}]_{u=0,\dots,p-1}$ as a
linear function of the vector
$\overline{\bS}_{n_j,j}$
defined by~(\ref{e:snjm}), up to a negligible term. We can then deduce, as in Section~\ref{sec:appl-wavel-estim}, the asymptotic behavior of $\hat d_0$ of the multivariate scalogram $\overline{\bS}_{n_j,j}$ using~(\ref{e:d0}).

We now indicate the asymptotic behavior of the univariate multiscale scalogram in the case $G=H_{q_0}$ since it will be needed in Section~\ref{sec:appl-wavel-estim}. We state the results separately for $q_0=1$ and for $q_0\geq2$.

We first consider the case $q_0=1$~:
\begin{theorem}\label{thm:multiscale:1}
Suppose $G=H_{q_0}$ with $q_0=1$ and that Assumptions~\textbf{A}(i),(ii) in Section~\ref{s:LRD} hold. Set
  $\gamma_j=2^j$ and let
  $\{(g_j)_{j\geq0},\,g_\infty\}$ be a sequence of univariate filters
  satisfying~\ref{ass:w-ap}--\ref{ass:w-c} with $m=1$ and $M\geq
  d+K$. Then, as $j\to\infty$,
  \begin{equation}
    \label{eq:wav-sp-asymp:1}
    {\sigma}^2_{j} \sim  f^*(0) \,
L_{1}(\widehat{g}_{\infty})\; 2^{2 j(d+K)}\;,
  \end{equation}
 where $L_1$ has been defined in~(\ref{e:defKp}). Let now $j=j(N)$ be an increasing sequence such that $j\to\infty$ and
  $N2^{-j}\to\infty$. Define $n_j$, $\widehat{\sigma}^2_{j}$ and
  ${\sigma}^2_{j}$ as in~(\ref{eq:nj_def}),~(\ref{eq:scalo_def})
  and~(\ref{eq:wav-sp_def}), respectively.  Then, as $N\to\infty$,
\begin{equation}
    \label{eq:scalo-multiscale-asymp:1}
   \left\{ n_j^{1/2}
     \left(\frac{\widehat{\sigma}^2_{j-u}}{{\sigma}^2_{j-u}}-1\right)\right\}_{u\geq0}
\overset{\mathrm{fidi}}{\longrightarrow}
\left\{
Q^{(d)}_u\right\}_{u\geq0} \;,
  \end{equation}
  where $Q^{(d)}$ denotes a centered Gaussian process with covariance function
  \begin{equation}
    \label{eq:cov:Q}
  \mathrm{Cov}(Q^{(d)}_u,Q^{(d)}_{u'})=\frac{4\pi \,2^{2(d+K)|u'-u|-\max(u,u')}}{L_{1}(\widehat{g}_{\infty})^2} \int_{-\pi}^{\pi} |D_{\infty,u-u'}(\lambda)|^2\rmd \lambda\;,
  \end{equation}
  with for all $m\in\mathbb{Z}$ and $\lambda \in (-\pi,\pi)$,
  \[
  D_{\infty,m}(\lambda)=\sum_{\ell\in\mathbb{Z}}|\lambda+2\pi\ell|^{-2(d+K)} \mathbf{e}_{m}(\lambda+2\pi\ell)\overline{\widehat{g}_\infty(\lambda+2\pi \ell)}
  \widehat{g}_\infty(2^{-m}(\lambda+2\pi \ell))\;,
  \]
  and
  \[
  \mathbf{e}_{m}(\xi)=2^{-m/2}[\rme^{-i2^{-m}v\xi},\,v=0,\cdots,2^m-1]^T\;.
  \]
\end{theorem}
\begin{proof}
We first observe that the proof of formula~(4.5) in Theorem~4.1 of~\cite{clausel-roueff-taqqu-tudor-2011b} remains valid in the case $q_0=1$. This yields~(\ref{eq:wav-sp-asymp:1}).

We now prove the convergence~(\ref{eq:scalo-multiscale-asymp:1}). To do so we adapt the corresponding proof of Theorem~4.1 of \cite{clausel-roueff-taqqu-tudor-2011b} done for $q_0\geq 2$. From~\cite{clausel-roueff-taqqu-tudor-2011b} (see equality~(9.5)), we have
\[
\widehat{\sigma}^2_{j-u}-{\sigma}^2_{j-u}=\frac{n_j}{n_{j-u}}\sum_{v=0}^{2^u-1}
\overline{S}_{n_j,j}(2^u+v)+O_P(\sigma_{j-u}^2/n_{j-u}),
\quad u=0,\dots,p-1\;,
\]
where we denoted the entries of the multivariate scalogram
$\overline{\bS}_{n_j,j}$ in~(\ref{e:snjm}) as
$[\overline{S}_{n_j,j}(\ell)]_{\ell=1,\dots,m}$. In addition, we also proved in Section~9 of~\cite{clausel-roueff-taqqu-tudor-2011b} that the multivariate filters $\bh_{j}(t)$ involved in the definition of the multivariate wavelet coefficients, defined
by~(\ref{eq:multiscale-filters}), satisfy the assumptions of Theorem~3.2 of~\cite{clausel-roueff-taqqu-tudor-2011b}. We can then apply Theorem~3.2 (a) of~\cite{clausel-roueff-taqqu-tudor-2011b} which provides the asymptotic behavior of the multivariate scalogram $\overline{S}_{n_j,j}$. Using the equality (9.6) of \cite{clausel-roueff-taqqu-tudor-2011b} relating $\widehat{h}_{\ell,\infty}$ and $\widehat{g}_\infty$ as,
\[
\widehat{h}_{\ell,\infty}(\lambda)=2^{-u/2}\widehat{g}_\infty(2^{-u}\lambda)\rme^{\rmi 2^{-u}v\lambda}\;,
\]
we then deduce that
as $j\to\infty$,
$$
\left\{n_j^{1/2} 2^{-2(j-u)(d+K)}
\overline{\bS}_{n_j,j}(2^u+v)\right\}_{u,v}
\overset{\mathrm{(\mathcal{L})}}{\longrightarrow}\mathcal{N}(0,\widetilde{\Gamma})\;,
$$
where (we denote $\lambda_p=\lambda+2p\pi$),
\begin{eqnarray*}
&&\widetilde{\Gamma}_{(u,v),(u',v')}=2^{2(u+u')(d+K)}\Gamma_{2^u+v,2^{u'}+v}\\
&=&4\pi (f^*(0))^2 \; 2^{2(u+u')(d+K-\frac{1}{2})}
\int_{-\pi}^\pi \left|
\sum_{p\in\mathbb{Z}}
|\lambda_p|^{-2(K+d)}
\widehat{g}_{\infty}(2^{-u}\lambda_p)\overline{\widehat{g}_{\infty}}(2^{-u'}\lambda_p)\rme^{\rmi(2^{-u}v-2^{-u'}v')\lambda_p}
\right|^2 \, \rmd\lambda\;,
\end{eqnarray*}
and $(u,v)$ (resp $(u',v')$) take values $u=0,\dots,p-1$ (resp $u'=0,\dots,p-1$) and $v=0,\dots,2^u-1$ (resp $v'=0,\dots,2^{u'}-1$). We showed in ~\cite{clausel-roueff-taqqu-tudor-2011b}, Relation~(9.4), that as $j\to\infty$, $n_j/n_{j-u}\sim 2^{-u}$. Using also~(\ref{eq:wav-sp-asymp:1}), which implies that $\sigma_{j-u}^2\sim f^*(0) \,
L_{1}(\widehat{g}_{\infty})\; 2^{2 (j-u)(d+K)}$ as $j\to\infty$, and following the proof of Theorem~4.1 of~\cite{clausel-roueff-taqqu-tudor-2011b}, we get
$$
\left\{n_j^{1/2}
\frac1{{\sigma}^2_{j-u}}
\frac{n_j}{n_{j-u}}\sum_{v=0}^{2^u-1}\overline{\bS}_{n_j,j}(2^u+v)\right\}_u\overset{\mathrm{(\mathcal{L})}}{\longrightarrow}\mathcal{N}(0,\overline{\Gamma})\;,
$$
with
\begin{align}\label{e:cov}
&\overline{\Gamma}_{u,u'}\\
&=\frac{2^{-u-u'}}{(f^*(0))^2L_{1}(\widehat{g}_{\infty})^2}\sum_{v=0}^{2^u-1}\sum_{v'=0}^{2^{u'}-1}\widetilde{\Gamma}_{(u,v),(u',v')}\nonumber\\
&=\frac{2^{2(u+u')(d+K-\frac{1}{2})}}{(f^*(0))^2L_{1}(\widehat{g}_{\infty})^2}\sum_{v=0}^{2^u-1}\sum_{v'=0}^{2^{u'}-1}\Gamma_{2^u+v,2^{u'}+v}\nonumber\\
&=\frac{4\pi 2^{2(u+u')(d+K-1)}}{L_{1}(\widehat{g}_{\infty})^2}
\int_{-\pi}^\pi \sum_{v=0}^{2^u-1}\sum_{v'=0}^{2^{u'}-1}
\left|\sum_{p\in\mathbb{Z}}\frac{\widehat{g}_{\infty}(2^{-u}\lambda_p)
\overline{\widehat{g}_{\infty}(2^{-u'}\lambda_p)}\rme^{\rmi(2^{-u}v-2^{-u'}v')\lambda_p}}{|\lambda_p|^{2(K+d)}}\right|^2 \rmd\lambda\;,
\end{align}
and where $u,u'=0,\dots,p-1$. Thereafter, we follow the same lines that in the proof of~\cite[Theorem~2]{roueff-taqqu-2009b}. Assume for example that $u'\geq u$. We have to estimate
\[
\sum_{v=0}^{2^u-1}\sum_{v'=0}^{2^{u'}-1}\left|\sum_{p\in\mathbb{Z}}
|\lambda_p|^{-2(K+d)}
\widehat{g}_{\infty}(2^{-u}\lambda_p)\overline{\widehat{g}_{\infty}(2^{-u'}\lambda_p)}\rme^{\rmi(2^{-u}v-2^{-u'}v')\lambda_p}
\right|^2\;,
\]
which reads $\sum\limits_{v'=0}^{2^{u'}-1}G_{u,u',v'}(\lambda)$ with
\[
G_{u,u',v'}(\lambda)=\sum_{v=0}^{2^{u}-1}\left|\sum_{p\in\mathbb{Z}}
\rme^{\rmi(2^{-u}v-2^{-u'}v')\lambda_p}g_{u,u'}(2^{-u}\lambda_p)
\right|^2
\]
where $g_{u,u'}(\xi)=|2^{u}\xi|^{-2(K+d)}
\widehat{g}_{\infty}(\xi)\overline{\widehat{g}_{\infty}(2^{u-u'}\xi)}$.
We now observe that $G_{u,u',v'}$ is a $2\pi$--periodic function and write $p=2^u q+r$ with $r\in \{0,\cdots,2^u-1\}$. Hence ($\lambda_p=\lambda_r+2^u q\times 2\pi$ and $\rme^{2\rmi\pi v}$, if $v$ is integer),
\begin{eqnarray*}
G_{u,u',v'}(\lambda)&=&\sum_{v=0}^{2^{u}-1}\left|\sum_{r=0}^{2^{u}-1}\rme^{\rmi(2^{-u}v-2^{-u'}v')\lambda_r}\sum_{q\in\mathbb{Z}}
\rme^{-\rmi 2^{u-u'}v' 2\pi q}g_{u,u'}(2^{-u}\lambda_r+2\pi q)
\right|^2\\
&=&\sum_{v=0}^{2^{u}-1}\left|\sum_{r=0}^{2^{u}-1}\rme^{\rmi 2^{-u}v\lambda_r}h_{u,u',v'}(2^{-u}\lambda_r)
\right|^2\;,
\end{eqnarray*}
with
\[
h_{u,u',v'}(\xi)=\sum_{q\in\mathbb{Z}}
\rme^{-\rmi 2^{u-u'}v' (2\pi q+\xi)}g_{u,u'}(\xi+2\pi q)\;.
\]
Hence
\[
G_{u,u',v'}(\lambda)=\sum_{v=0}^{2^{u}-1}\sum_{r=0}^{2^{u}-1}\sum_{r'=0}^{2^{u}-1}\rme^{\rmi 2^{-u}v 2\pi(r-r')}h_{u,u',v'}(2^{-u}\lambda_r)\overline{h_{u,u',v'}(2^{-u}\lambda_{r'})}\;.
\]
Observe that if $r\neq r'$
\[
\sum_{v=0}^{2^{u}-1}\rme^{\rmi 2^{-u}v 2\pi(r-r')}=0\;,
\]
whereas in the case $r=r'$ this sum  equals $2^{u}$. Hence
\[
G_{u,u',v'}(\lambda)=2^u \,\sum_{r=0}^{2^{u}-1}|h_{u,u',v'}(2^{-u}\lambda_r)|^2\;.
\]
As in the proof of \cite[Theorem~2]{roueff-taqqu-2009b}, we apply Lemma~1 of~\cite{roueff-taqqu-2009} with $g=|h_{u,u',v'}|^2$, $\gamma=2^u$ and get
\[
\int_{-\pi}^\pi G_{u,u',v'}(\lambda)\rmd \lambda=2^u \int_{-\pi}^\pi\left(\sum_{r=0}^{2^{u}-1}|h_{u,u',v'}(2^{-u}\lambda_r)|^2\right)\rmd \lambda=2^{2u}\int_{-\pi}^\pi |h_{u,u',v'}(\lambda)|^2\rmd \lambda\;.
\]
We then deduce that
\[
\sum_{v'=0}^{2^{u'}-1}\int_{-\pi}^\pi G_{u,u',v'}(\lambda)\rmd \lambda=2^{2u}\int_{-\pi}^\pi \left(\sum_{v'=0}^{2^{u'}-1} |h_{u,u',v'}(\lambda)|^2\right)\rmd \lambda\;.
\]
Using~(\ref{e:cov}), the definition of $G_{u,u',v'}$ and the last display, we deduce that
\begin{eqnarray*}
&&\overline{\Gamma}_{u,u'}\\
&=&\frac{4\pi 2^{2(u+u')(d+K-1)}}{L_{1}(\widehat{g}_{\infty})^2}\left(2^{2u}\int_{-\pi}^\pi \sum_{v'=0}^{2^{u'}-1} |h_{u,u',v'}(\lambda)|^2\right)\\
&=&\frac{4\pi 2^{2(u+u')(d+K-1)}}{L_{1}(\widehat{g}_{\infty})^2}\times 2^{2u}2^{-4u(d+K)}\sum_{v'=0}^{2^{u'}-1}\int_{-\pi}^\pi \left| \sum_{q\in\mathbb{Z}}|\lambda_q|^{-2(d+K)}\rme^{-\rmi 2^{u-u'}v'\lambda_q}\widehat{g}_\infty(\lambda_q)\overline{\widehat{g}_\infty(2^{-(u-u')}\lambda_q)}\right|^2\rmd \lambda\\
&=&\frac{4\pi 2^{2(u'-u)(d+K)-2u'}}{L_{1}(\widehat{g}_{\infty})^2}\sum_{v'=0}^{2^{u'}-1}\int_{-\pi}^\pi \left| \sum_{q\in\mathbb{Z}}|\lambda_q|^{-2(d+K)}\rme^{-\rmi 2^{u-u'}v'\lambda_q}\widehat{g}_\infty(\lambda_q)\overline{\widehat{g}_\infty(2^{-(u-u')}\lambda_q)}\right|^2\rmd \lambda\;.
\end{eqnarray*}
For $v'\in\{0,\cdots,2^{u'}-1\}$, we write $v'=v+k2^{u'-u}$ with $v\in \{0,\cdots,2^{u'-u}-1\}$ and $k\in \{0,\cdots,2^u-1\}$ and transform the sum in $v'$ into a sum over $v$ and $k$. We obtain
\[
\overline{\Gamma}_{u,u'}=\frac{4\pi 2^{2(u'-u)(d+K)-2u'}}{L_{1}(\widehat{g}_{\infty})^2}\sum_{v=0}^{2^{u'-u}-1}\sum_{k=0}^{2^{u}-1}\int_{-\pi}^\pi \left| \sum_{q\in\mathbb{Z}}|\lambda_q|^{-2(d+K)}\rme^{-\rmi 2^{u-u'}(v+k2^{u'-u})\lambda_q}\widehat{g}_\infty(\lambda_q)\overline{\widehat{g}_\infty(2^{-(u-u')}\lambda_q)}\right|^2\rmd \lambda\;.
\]
Since $\rme^{-\rmi 2^{u-u'}v'\lambda_q}=\rme^{-\rmi 2^{u-u'}v\lambda_q}\rme^{-\rmi k\lambda}$ and $\sum_{k=0}^{2^{u}-1}|\rme^{-\rmi k\lambda}|^2=2^u$, one has
\[
\overline{\Gamma}_{u,u'}=\frac{4\pi 2^{2(u'-u)(d+K)-u'}2^{u-u'}}{L_{1}(\widehat{g}_{\infty})^2}\sum_{v=0}^{2^{u'-u}-1}\int_{-\pi}^\pi \left| \sum_{q\in\mathbb{Z}}|\lambda_q|^{-2(d+K)}\rme^{-\rmi 2^{u-u'}v\lambda_q}\widehat{g}_\infty(\lambda_q)\overline{\widehat{g}_\infty(2^{-(u-u')}\lambda_q)}\right|^2\rmd \lambda\;.
\]
Define now for any $m\in\mathbb{Z}$, the vector
\[
\mathbf{e}_m(\xi)=2^{-m/2}[\rme^{\rmi 2^{-m}v\xi},v=0,\cdots,2^m-1]^T\;.
\]
We then recover~(\ref{eq:cov:Q}) which concludes the proof.
\end{proof}
The case $q_0\geq2$ has been considered in~\cite[Theorem~4.1]{clausel-roueff-taqqu-tudor-2011b}. We recall it here.
\begin{theorem}\label{thm:multiscale:2}
  Suppose $G=H_{q_0}$, $q_0\geq 2$ and that Assumptions~\textbf{A}(i),(ii) hold with
  $q_0\geq2$. Set
  $\gamma_j=2^j$ and let
  $\{(g_j)_{j\geq0},\,g_\infty\}$ be a sequence of univariate filters
  satisfying~\ref{ass:w-ap}--\ref{ass:w-c} with $m=1$ and $M\geq
  \delta(q_0)+K$. Then, as $j\to\infty$,
  \begin{equation}
    \label{eq:wav-sp-asymp:2}
    {\sigma}^2_{j} \sim q_0!  \, (f^*(0))^{q_0} \,
L_{q_0}(\widehat{g}_{\infty})\; 2^{2 j (\delta(q_0)+K)}\;,
  \end{equation}
  where $L_p$ has been defined in~(\ref{e:defKp}) for any $p\geq 1$. Let now $j=j(N)$ be an increasing sequence such that $j\to\infty$ and
  $N2^{-j}\to\infty$. Define $n_j$, $\widehat{\sigma}^2_{j}$ and
  ${\sigma}^2_{j}$ as in~(\ref{eq:nj_def}),~(\ref{eq:scalo_def})
  and~(\ref{eq:wav-sp_def}), respectively.  Then, as $N\to\infty$,
\begin{equation}
    \label{eq:scalo-multiscale-asymp:2}
   \left\{ n_j^{1-2d}
     \left(\frac{\widehat{\sigma}^2_{j-u}}{{\sigma}^2_{j-u}}-1\right)\right\}_{u\geq0}
\overset{\mathrm{fidi}}{\longrightarrow}
\left\{2^{(2d-1)u}\,
\frac{L_{q_0-1}(\widehat{g}_{\infty})}{q_0!\,L_{q_0}(\widehat{g}_{\infty})}
\,Z_d(1)\right\}_{u\geq0} \;.
  \end{equation}
\end{theorem}

\noindent\textbf{Acknowledgments}

Marianne Clausel's research was partially supported by the PEPS project
\emph{AGREE} and LabEx \emph{PERSYVAL-Lab} (ANR-11-LABX-0025-01) funded by the
French program Investissement d'avenir. François Roueff's research was
partially supported by the ANR project \emph{MATAIM} NT09 441552. Murad~S.Taqqu
was supported in part by the NSF grants DMS--1007616 and DMS-1309009 at Boston
University.

\bibliographystyle{apalike}
\bibliography{lrd}
\end{document}